\documentclass[final]{siamltex}

\usepackage{amsmath}
\usepackage{amssymb}
\usepackage{graphicx}

\newcommand{\p}{\partial}

\renewcommand{\phi}{\varphi}

\renewcommand{\a}{\alpha}

\newcommand{\eb}{\mathbf{e}}

\newcommand{\R}{{\mathbb R}}

\newcommand{\eps}{\varepsilon}

\newtheorem{remark}{Remark}

\title{Fokker-Planck Equations for Stochastic Dynamical Systems with Symmetric L\'evy Motions}

\author{T. Gao\footnotemark[1]
\and
J. Duan\footnotemark[2]\ \footnotemark[4]
\and
X. Li\footnotemark[3]}

\begin{document}

\maketitle

\renewcommand{\thefootnote}{\fnsymbol{footnote}}

\footnotetext[1]{Department of Applied Mathematics,
Illinois Institute of Technology, Chicago, IL 60616
({\tt tinggao0716@gmail.com}).}
\footnotetext[2]{Department of Applied Mathematics,
Illinois Institute of Technology, Chicago, IL 60616
({\tt duan@iit.edu}).}
\footnotetext[3]{Corresponding author. Department of Applied Mathematics,
Illinois Institute of Technology, Chicago, IL 60616
({\tt lix@iit.edu}).}
\footnotetext[4]{Institute for Pure and Applied Mathematics,
University of California, Los Angeles, CA 90095}

\renewcommand{\thefootnote}{\arabic{footnote}}

\begin{abstract}
 The Fokker-Planck equations for stochastic dynamical systems, with non-Gaussian $\alpha-$stable symmetric L\'evy motions,  have a nonlocal  or fractional Laplacian term.  This nonlocality is the manifestation of the effect of non-Gaussian fluctuations. Taking advantage of the Toeplitz matrix structure of the time-space discretization, a fast and accurate numerical algorithm is proposed to simulate the nonlocal Fokker-Planck equations, under either absorbing or natural conditions. The scheme is shown to satisfy a discrete maximum principle and to be convergent. It is validated against a known exact solution and the numerical solutions obtained by using other methods. The numerical results for two prototypical stochastic systems, the Ornstein-Uhlenbeck system and the double-well system are shown.  
\end{abstract}



\pagestyle{myheadings}
\thispagestyle{plain}
\markboth{GAO, DUAN AND LI }{NONLOCAL FOKKER-PLANCK EQUATIONS}

\section{Introduction}  \label{intro}
%%%%%%%%%%%%%%%%%%%%%%%%%%%%%%%%%%%%%%%%%%%%%%%%%%%%%%%%%%%%%%%%%

  The Fokker-Planck (FP) equations for stochastic dynamical systems
  \cite{Gardiner2009,Gardiner2004} describe the time evolution for the probability density of solution paths.
  When the noise in a system is Gaussian (i.e., in terms of Brownian motion), the corresponding Fokker-Planck equation is a differential equation. The Fokker-Planck equations are widely used to investigate stochastic dynamics in physical, chemical and biological systems.

  However,   the noise is often non-Gaussian \cite{Woy}, e.g., in terms of $\alpha-$stable L\'evy motions. The corresponding Fokker-Planck equation has an   integral term and is a differential-integral equation. In fact, this extra term  is an integral over the whole state space and thus represents a nonlocal effect caused by the non-Gaussianity of the noise.

  A Brownian motion is a Gaussian stochastic process, characterized by its mean or drift vector (taken to be zero for convenience), and a diffusion or covariance matrix.

A L\'evy process  $L_t$  on $\mathbb{R}^n$ is   a non-Gaussian
stochastic process. It is characterized by a drift vector $b
\in\mathbb{R}^n$ (taken to be zero for convenience), a Gaussian
covariance matrix $A$, and a non-negative Borel measure $\nu$,
defined on $(\R^n, \mathcal{B}(\R^n))$ and concentrated on $\R^n
\setminus\{0\}$. The measure $\nu$  satisfies the condition
\begin{equation*}
  \int_{\R^n \setminus\{0\} } (y^2 \wedge 1) \; \nu(dy) < \infty,
\end{equation*}
where $a \wedge b =\min\{a, b\}$, or equivalently
\begin{equation*}
  \int_{\R^n \setminus\{0\} } \frac{y^2}{1+y^2}\; \nu(dy) < \infty.
\end{equation*}
This measure $\nu$ is the so called    L\'evy jump measure of the
L\'evy process $L_t$. We also call $(b, A, \nu)$ the
\emph{generating triplet} \cite{Applebaum2009, Sato-99}.

In this paper, we   consider stochastic differential equations
(SDEs) with a special class of L\'evy processes, the $\alpha$-stable
symmetric L\'evy motions. The corresponding Fokker-Planck equations
(for the evolution of the probability density function of the
solution) contain a nonlocal term, i.e., the fractional Laplacian
term, which quantifies the non-Gaussian effect. It is hardly
possible to have analytical solutions for these nonlocal
Fokker-Planck equations even for  simple systems. We thus consider
numerical simulation for these nonlocal equations, either on a
bounded domain or a unbounded domain.  Due to the nonlocality,
however, the    `boundary' condition needs to be prescribed on the
entire exterior domain. A recently developed nonlocal vector calculus
has provided sufficient conditions for the well-posedness 
of the initial-boundary value problems on this type of nonlocal diffusion 
equations with volume constraints. See the review paper \cite{du2011analysis}. 

A few authors have considered numerical simulations of
reaction-diffusion type partial differential equations with a
(formal) fractional Laplacian operator. They impose a boundary
condition only on the boundary of a bounded set (not on the entire
exterior set as we do here);  see \cite{ervin2007numerical},
\cite[Eqs.~(10) and (11)]{li2012finite}, \cite[\S 3]{bueno2012fourier}, \cite{meerschaert2004finite},
\cite[Eq.~(2)]{yang2010numerical} and references therein. The domain of
definition of the  fractional Laplacian operator in these papers
consists of certain functions with prescribed values on the
boundary. However, the domain of definition of the  fractional
Laplacian operator in our present  paper consists of certain
functions with prescribed values on the entire exterior domain.
Thus,  the fractional Laplacian operator in the present paper is
different from that in \cite{ervin2007numerical, li2012finite, bueno2012fourier, meerschaert2004finite, yang2010numerical}, as
the domains of definition are different. Note that the exterior
boundary condition is required for understanding probability density
evolution. Recently, discontinuous and continuous Galerkin methods 
have been developed for the volume-constrained nonlocal diffusion 
problems and their error analysis is given in \cite{du2011analysis} and the
references therein. To our knowledge, the work \cite{Du13} is closely
related to our work, where the authors of \cite{Du13} have shown
the well-posedness, the maximum principle, conservation and dispersion
relations for a class of nonlocal convection-diffusion equations with
volume constraints. 
A finite difference scheme is also presented in \cite{Du13}, which
maintains the maximum principle when suitable conditions are met. 
Compared with this paper, the difference is that the jump processes 
considered in \cite{Du13} are nonsymmetric and of finite-range.

This paper is organized as follows. In section \ref{astable}, we
present the nonlocal Fokker-Planck equation for a SDE with
$\alpha-$stable symmetric L\'evy motion and then devise a numerical
discretization scheme, with either the absorbing or natural 
condition. For simplicity, we consider scalar SDEs. The description
and analysis of the numerical schemes are presented in section~\ref{sec.nm} 
and section~\ref{sec.na}, respectively. The numerical experiments 
are conducted in section~\ref{sec.nr}. The paper ends with some discussions
in section~\ref{sec.cl}.

\section{Fokker-Planck equations for SDEs with L\'evy motions}
\label{astable}

Consider a scalar SDE
\begin{equation} \label{sde999}
 {\rm d}X_t = f(X_t)\,{\rm d}t + {\rm d}L_t, X_0= x_0,
\end{equation}
where $f$ is a given deterministic vector field,  the scalar L\'evy
process $L_t$ has the generating triplet $(0, d, \eps \nu_\a)$, with
diffusion constant $d\geq 0$ and the $\alpha-$stable symmetric jump
measure
 $$
\nu_\a({\rm d}y)=C_\alpha|y|^{-(1+\alpha)}\, {\rm d}y.
$$
The constant $C_\alpha$ is defined as $\displaystyle{C_{\alpha} =
\frac{\alpha}{2^{1-\alpha}\sqrt{\pi}}
\frac{\Gamma(\frac{1+\alpha}{2})}{\Gamma(1-\frac{\alpha}{2})}}$. The
non-Gaussianity index $\alpha \in (0, 2)$ and the intensity constant
is $\eps$. For more information on $\alpha-$stable L\'evy motions,
see \cite{wu, Chen, Schertzer}.

The FP equation for the distribution of the conditional probability
density $p(x,t) = p(x, t|x_0, 0)$, i.e., the probability of the
process $X_t$ has value $x$ at time $t$ given it had value $x_0$ at
time $0$, is given by \cite{Applebaum2009, Schertzer, wu, wu2}
\begin{equation} \label{ffp3}
 p_t = - (f(x)\; p)_x + \frac12   d  p_{xx}
 + \eps [ k_{\a}  (-\Delta)^{\frac{\alpha}{2}}]^* p.
\end{equation}
Note that the integro-differential part $  k_{\a}
(-\Delta)^{\frac{\alpha}{2}}$ is  symmetric on $L^2(\R)$ (see
\cite{wu}).   Thus,
\begin{equation} \label{ffp4}
 \p_t p = -\p_x [f(x) \; p] + \frac{d}2      \p_{xx} p
 +  \eps k_{\a}  (-\Delta)^{\frac{\alpha}{2}} p,
\end{equation}
where
$$
k_{\alpha} \triangleq  \int_{\R \setminus\{0\}} (\cos y -1 ) \;
\nu_{\a}(dy) <0.
$$

This equation may be rewritten   as
\begin{eqnarray} \label{FPE555}
  p_t &=& - (f(x)\; p)_x + \frac12   d  p_{xx}   \nonumber \\
   &+& \int_{\R \setminus\{0\}} [p(x+  y, t)-p(x, t) -   I_{\{|y|<1\}} \; y \; p_x(x, t) ] \; \eps \nu_\a({\rm d}y) .
\end{eqnarray}

\subsection{Auxiliary Conditions}

We consider two kinds of auxiliary conditions, the absorbing nonlocal
"Dirichlet" condition and natural far-field condition.

The absorbing condition on the finite interval $(a, b)$ states that the
probability of finding ``particles" $X_t$ outside $(a, b)$ is zero:
\begin{equation}
 p(x, t)=0, \;\; \mbox{  for } x \notin (a, b).
\label{eq.abc}
\end{equation}

%Question 1: \\ Suppose $p(x,t)$ is the solution of equation
%\eqref{FPE555}, given initial condition $p(x,0)\geq 0$ for all
%$x\in(a,b)$. Under what conditions on $f(x)$,$d$ and $\alpha$, will
%we have $\forall t>0$, $p(x,t)\geq 0$for all $x\in(a,b)$?

%I have tried cases: $d=0.1$. \\
%(1) when $\alpha<1$, $\varepsilon =0$or $0.1$, $f(x)=0$, or
%$f(x)=x$, or $f(x)=x-x^3$, these cases are all stable. However, if
%$f(x)=-x$, the graph of $p(x,t)$ will not be nonnegative at the
%first several steps, but later it will keep to be nonnegative all
%the time.\\
%(2) when $\alpha>1$, we have big oscillations.\\
%
%
%Question 2: \\
%First, what's the impact of $\alpha$ on the smoothness of solution $p(x,t)$? \\
%Second, since $0$ is the unstable equilibrium in $f(x)=x$ and
%$f(x)=x-x^3$, but the stable equilibrium in $f(x)=-x$, will this
%affect the numerical solutions of $p(x,t)$?

%Question 3: \\ Suppose $p(x,t)$ is the solution of equation
%\eqref{FPE555}, given initial condition $p(x,0)\geq 0$ for all $x\in
%\mathbb{R}$. Under what conditions on $f(x)$,$d$ and $\alpha$, will
%we have $\forall t>0$, $p(x,t)\geq 0$for all $x\in \mathbb{R}$?
%Could we get some answer from the paper \cite{Abels}?

%\subsubsection{Natural boundary conditions}
The natural far-field conditions in this work are specified as
follows: the domain is the whole real line $\R=(-\infty,\infty)$,
\begin{equation}
p(x,t)\rightarrow 0, \quad \text{ as } |x| \rightarrow \infty,
\label{eq.nbc}
\end{equation}
and the decay rate is fast enough to ensure
\begin{equation}
\int_{-\infty}^{\infty} p(x,t)\, {\rm d}x =1, \quad \text{for all }
t\ge 0. \label{eq.cm}
\end{equation}
%In one-dimensional case, $p(x,t)$ decays to zero faster than $1/|x|$
%as $|x| \rightarrow \infty$.

\section{Numerical Methods}
\label{sec.nm}

%We apply the Strang splitting method to our system in Equation
%(\ref{ffp4}). With this approach, the advection term is computed
%once and the rest terms consisting of both diffusion and
%integral-differential part (fractional Laplacian operator) are
%computed twice at each time step:
%\begin{equation}
% \begin{cases}
%P_t^* &= \frac{d}{2} P_{xx}^* + L(P^*) \ \ \mathrm{on}\
%(t_n,t_{n+1/2}), \ \ P^*(t_n) = P(t_n)\\
%P_t^{**} &= -(f(x)P^{**})_x  \ \ \mathrm{on}\
%(t_n,t_{n+1}), \ \ P^{**}(t_n) = P^*(t_{n+\frac{1}{2}})\\
%P_t^{***} &= \frac{d}{2}P_{xx}^{***} + L(P^{***})\ \mathrm{on}\
%(t_{n+1/2},t_{n+1}),\ P^{***}(t_{n+\frac{1}{2}}) = P^{**}(t_{n+1})
%\end{cases}
%\label{splitting}
%\end{equation}
%where we denote $L(p):=\eps k_{\a}  (-\Delta)^{\frac{\alpha}{2}} p$
%and $P(t_{n+1}) = P^{***}(t_{n+1})$ . \\

For clarity and without losing generality, in the following, we
present the numerical algorithms the case of $(a,b)=(-1,1)$ with the
absorbing condition \eqref{eq.abc}. 

 Noting that the absorbing condition dictates that
$p$ vanish outside $(-1,1)$, we can simplify Eq.~(\ref{FPE555}) by
decomposing the integral into three parts
$\int_{\R}=\int_{-\infty}^{-1-x} + \int_{-1-x}^{1-x} +
\int_{1-x}^{\infty}$ and analytically evaluating the first and third
integrals
\begin{eqnarray}
 p_t &=&  -( f(x) p)_x + \frac{d}{2} p_{xx}
  - \frac{\eps C_\a}{\a} \left[\frac{1}{(1+x)^\a}+\frac{1}{(1-x)^\a}\right] p
\nonumber \\
 && + \eps C_\a \int_{-1-x}^{1-x} \frac{p(x+y,t) - p(x,t)}{|y|^{1+\a}}\; {\rm d}y,
\label{fpe1Dn3}
\end{eqnarray}
for $x \in (-1,1)$; and $p(x,t)=0$ for $x \notin  (-1, 1)$. We have
dropped the integral $\displaystyle{\int_\R \frac{I_{\{ |y|<1\}}(y)
\, y} {|y|^{1+\alpha}}\; {\rm d}y}$, since the principal value
integral vanishes and we find it does not improve the accuracy of
the numerical integration of the singular integral in
\eqref{FPE555} \cite{Ting12}. 

Next, we describe a numerical method for discretizing the
Fokker-Planck equation \eqref{fpe1Dn3}, a time-dependent
integro-differential equation. For the spatial derivatives, the
advection term $-(f(x) p)_x$ is discretized by the third-order WENO
method provided in \cite{jiang1996efficient} 
and \cite{jiang2000weighted}) while the
diffusion term is approximated by the second-order central
differencing scheme. The numerical scheme for the remaining nonlocal
integral term in \eqref{fpe1Dn3} is presented in detail as follows.

   Let's divide the interval $[-2,2]$ in space into $4J$ sub-intervals
and define $x_j=jh$ for $-2J\leq j \leq 2J$ integer, where $h=1/J$.
We denote the numerical solution of $p$ at $(x_j,t)$ by $P_j$. Then,
using a modified trapezoidal rule for approximating the singular
integral \cite{Ting12}, we obtain the semi-discrete equation
\begin{equation}
  \begin{split}
  \dfrac{{\rm d}P_j}{{\rm d}t} = & C_h \frac{P_{j-1} - 2P_j + P_{j+1}}{h^2}
   - [(fP)_{x,j}^+ +  (f P)_{x,j}^-] \\
   &  -  \frac{\eps C_\a}{\a} \left[\frac{1}{(1+x_j)^\a}+\frac{1}{(1-x_j)^\a}\right] P_j
    + \eps C_\a h \sum^{J-j}_{k=-J-j,k\neq 0}\!\!\!\!\!\!\!\!\!{''} \;
    {\frac{P_{j+k} - P_j}{|x_k|^{1+\alpha}} },
  \end{split}
 \label{nm1D3}
\end{equation}
for $0<\a<2$ and $j = -J+1, \cdots, -2,-1,0,1,2, \cdots, J-1$, where
the constant $C_h$ $\displaystyle{ = \frac{d}{2} - \eps C_\a
\zeta(\alpha-1) h^{2-\a}}$. $\zeta$ is the Riemann zeta function.
The $\pm$ superscripts denote the global
Lax-Friedrichs flux splitting defined as $(fP)^{\pm} =
\frac{1}{2}(fP \pm \alpha P)$ with $\alpha=\max_x |f(x)|$. The
summation symbol $\sum{''}$ means that the quantities corresponding
to the two end summation indices are multiplied by $1/2$. The
absorbing condition requires that the values of $P_j$ vanish if the
index $|j|\geq J$.

The first-order derivatives $(fP)_{x,j}^+$ and $(f P)_{x,j}^-$
are approximated by the third-order WENO scheme \cite{jiang2000weighted}. 
%For
%$j=-J+1, \cdots, -2,-1,0,1,2, \cdots, J-1$, denote
%$\varphi_{j}^{\pm} = \frac{1}{2}(f(x_j)P_j \pm \alpha P_j)$ with
%$\alpha=\max|f(x)|$. 
The approximation to $\varphi_x(x_j)$ on
left-biased stencil $\{ x_k, k=j-2, j-1, j, j+1\}$ is given by
$$ \varphi_{x,j}^+ =\frac{1}{2h}(\Delta^+ \varphi_{j-1} + \Delta^+ \varphi_{j} )
- \frac{\omega_-}{2h}(\Delta^+ \varphi_{j-2} - 2\Delta^+
\varphi_{j-1}  +\Delta^+ \varphi_{j})$$ where
$$ \omega_-= \frac{1}{1+2r_-^2}, \ \ \ \ r_- = \frac{\delta + (\Delta^-\Delta^+ \varphi_{j-1})^2}
{\delta + (\Delta^-\Delta^+ \varphi_{j})^2}.$$ 
Here $\delta=10^{-6}$ is used to prevent the denominators from being zero.
The approximation to
$\varphi_x(x_j)$ on right-biased stencil $\{ x_k, k=j-1, j, j+1, j+2\}$
is given by
$$\varphi_{x,j}^- =\frac{1}{2h}(\Delta^+ \varphi_{j-1} + \Delta^+ \varphi_{j} )
- \frac{\omega_+}{2h}(\Delta^+ \varphi_{j+1} - 2\Delta^+
\varphi_{j} +\Delta^+ \varphi_{j-1})$$ where
$$ \omega_+= \frac{1}{1+2r_+^2}, \ \ \ \ r_+ = \frac{\delta+ (\Delta^-\Delta^+ \varphi_{j+1})^2}
{\delta + (\Delta^-\Delta^+ \varphi_{j})^2}.$$ 
Note that $\Delta^+
\psi_j = \psi_{j+1} -\psi_{j}$, $\Delta^- \psi_j = \psi_{j}
-\psi_{j-1}$  and
$\varphi_{-J-1}^{\pm}=\varphi_{-J}^{\pm}=\varphi_{J}^{\pm}=\varphi_{J+1}^{\pm}=0.
$

%%%%%%%%%%%%%%%%%%%%%%%%%%%%%%%%%%%%%%%%%%%%%%%%%%%%%%%%%%%%%%%%%%%%%%
For time evolution, we adopt a third-order total variation
diminishing(TVD) Runge-Kutta method given in \cite{shu1988efficient}. In
particular, for the ordinary differential equation(ODE) $\dfrac{{\rm
d}u}{{\rm d}t}= R(u)$, the method can be written as
\begin{equation}
\begin{split}
U^{(1)} &= U^n + \Delta t R(U^n),\\
U^{(2)} &= \frac{3}{4}U^n + \frac{1}{4} U^{(1)} +\frac{1}{4}\Delta t
R(U^{(1)}),\\
U^{n+1} &= \frac{1}{3} U^n + \frac{2}{3} U^{(2)} + \frac{2}{3}
\Delta t R(U^{(2)}),
\end{split}
\label{eq.rk3}
\end{equation}
where $U^n$ denotes the numerical solution of $u$ at time $t=t_n$.

%%%%%%%%%%%%%%%%%%%%%%%%%%%%%%%%%%%%%%%%%%%%%modified in Aug.2013
The computational cost of the proposed numerical method 
is dominated by the summation term in the scheme \eqref{nm1D3}.
If the right-hand side of \eqref{nm1D3} is written in matrix-vector
multiplication form, the matrix corresponding to the summation term 
would be dense but is Toeplitz. The Toeplitz-vector products can be computed
fast with $O(J\log J)$ number of operation counts instead of $O(J^2)$,
where the matrix dimension is $O(J)$ \cite{golub2012matrix}.  
We find that the fast algorithm is crucial in solving the Fokker-Planck
equation for the natural far-field condition.

%%%%%%%%%%%%%%%%%%%%%%%%%%%%%%%%%%%%%%%%%%%%%modified in Aug.2013
\begin{remark}

For the natural far-field condition \eqref{eq.nbc}, the semi-discrete
equation corresponding to \eqref{FPE555} becomes
\begin{equation}
  \begin{split}
  \dfrac{{\rm d}P_j}{{\rm d}t} = & C_h \frac{P_{j-1} - 2P_j + P_{j+1}}{h^2}
   - [(fP)_{x,j}^+ +  (f P)_{x,j}^-]
    + \eps C_\a h \sum^{J-j}_{k=-J-j,k\neq 0}\!\!\!\!\!\!\!\!\!{''} \;
    {\frac{P_{j+k} - P_j}{|x_k|^{1+\alpha}} },
  \end{split}
 \label{eq.sdn}
\end{equation}
where $J=L/h$ and $L\gg 1$. 
Unlike the case of the absorbing condition, 
the values of the density outside the computational domain $(-L,L)$
are unknown for the natural condition. 
We perform numerical
calculations on increasing larger intervals $(a,b)=(-L,L)$ until the
results are convergent.
\end{remark}

%\begin{remark}  For absorbing condition, discontinuous
%initial condition like uniform distribution may lead to
%discontinuous or nearly discontinous probability density 
%at the boundaries when $\alpha<1$. 
%Therefore, for the computational points next to the boundary, 
%we use one-sided second-order difference schemes for spatial
%derivatives. 

%numerical differentiation: $f'(x)=(3f(x)-4f(x+h) +f(x+2h))/2h
%= (-3f(x)+ 4f(x-h) -f(x-2h))/2h $ and $f''(x)=(2f(x)- 5f(x\pm h)
%+4f(x\pm 2h) -f(x\pm 3h) ) /(2h^2)$.
%\end{remark}

\section{Analysis of our numerical scheme}
\label{sec.na}

The numerical analysis for linear convection-diffusion equations
are well-known and have been described in many classic textbooks. 
In this section, we perform the analysis on the nonlocal term 
in \eqref{FPE555} due to the $\alpha$-stable process 
by requiring neither Gaussian diffusion 
nor the deterministic drift be present. i.e., $d=0$ and $f\equiv 0$. 
We now analyze the numerical schemes proposed in \S \ref{sec.nm},
by examining its discrete maximal principle,  stability and
convergence.

\subsection{Discrete Maximum Principle}\label{sec.mp}

First, we examine the discrete maximum principle of the scheme \eqref{nm1D3}
with the absorbing nonlocal Dirichlet condition \eqref{eq.abc}. Due to the positivity
of the probability density function and the vanishing condition \eqref{eq.abc}
outside the domain, the maximum principle states that the solution in the domain
lies within the interval $[0,M]$, where $M$ is the maximum value
of the initial probability function.

\begin{proposition} \label{prop1}
When both $d$ and $f$ vanish, 
the scheme \eqref{nm1D3} corresponding to the absorbing 
condition with forward Euler for time integration
satisfies the discrete maximum principle,
i.e., $0\leq P_j^0\leq M$ for all $j$'s implies $0\leq P_j^{n}\leq M$
for all $n>0$ and $j$'s,
if the following inequality holds
\begin{equation}
\frac{\Delta t}{h^{\a}} \leq \frac{1}{2 \eps C_{\a} \left[1+
\frac{1}{\a} - \zeta(\a -1) \right]}. 
\label{eq.cdm}
\end{equation} 
\end{proposition}

\begin{proof}
We only need to show $0\leq P_j^{n+1}\leq M$ given $0\leq P_j^{n}\leq M$.
Applying the explicit Euler to the semi-discrete scheme \eqref{nm1D3},
we have
\begin{equation}
  \begin{split}
 P_j^{n+1}=&P_j^{n} - \Delta t \eps C_\a \zeta(\alpha-1) h^{-\a}(P_{j-1}^n -2 P_j^n +P_{j+1}^n) \\
 & -\frac{\Delta t\eps C_\a}{\a} 
    \left[\frac{1}{(1+x_j)^\a}+\frac{1}{(1-x_j)^\a}\right] P_j^n 
   +  \Delta t \eps C_\a h \sum^{k=J-j}_{k=-J-j,k\neq
0}\!\!\!\!{''} \;
    {\frac{P_{j+k}^n -P_j^n }{|x_k|^{1+\a}} }  \\
   = &  \left\{1+ \frac{2\Delta t \eps C_\a \zeta(\alpha-1)}{ h^{\a}}
     -\frac{\Delta t\eps C_\a}{\a}
     \left[\frac{1}{(1+x_j)^\a}+\frac{1}{(1-x_j)^\a}\right]\right. \\
   & -\Delta t \eps C_\a h 
     \left. \sum^{k=J-j}_{k=-J-j,k\neq 0}\!\!\!\!{''} \;
    {\frac{1}{|x_k|^{1+\a}} }\right\} P_j^n
    -\frac{\Delta t \eps C_\a \zeta(\alpha-1)}{ h^{\a}} 
            \left(P_{j-1}^n+P_{j+1}^n\right)  \\
   & + \Delta t \eps C_\a h \sum^{k=J-j}_{k=-J-j,k\neq
0}\!\!\!\!{''} \;
    {\frac{P_{j+k}^n}{|x_k|^{1+\a}} },
  \end{split}
\label{eq.fe}
\end{equation}
where $J=1/h$.
Noting $\zeta(x)$ is negative for $-1<x<1$, the coefficients of non-diagonal
terms $P_m^n$ with $m\neq j$ are all positive. 
Thus,
\begin{equation}
  \begin{split}
    P_j^{n+1} \leq & 
      \left\{1+ \frac{2\Delta t \eps C_\a \zeta(\alpha-1)}{ h^{\a}}
     -\frac{\Delta t\eps C_\a}{\a}
     \left[\frac{1}{(1+x_j)^\a}+\frac{1}{(1-x_j)^\a}\right]\right. \\
    & -\Delta t \eps C_\a h 
    \left. \sum^{k=J-j}_{k=-J-j,k\neq
0}\!\!\!\!{''} \;
    {\frac{1}{|x_k|^{1+\a}} }\right\}M \\
   & +\left[-2\frac{\Delta t \eps C_\a \zeta(\alpha-1)}{ h^{\a}} + \Delta t \eps C_\a h \sum^{k=J-j}_{k=-J-j,k\neq
0}\!\!\!\!{''} \;
    {\frac{1}{|x_k|^{1+\a}} }\right]M \leq M,
  \end{split}
\end{equation}
provided that 
\begin{equation} 
 1+ \frac{2\Delta t \eps C_\a \zeta(\alpha-1)}{ h^{\a}}
     -\frac{\Delta t \eps C_\a}{\a}
     \left[\frac{1}{(1+x_j)^\a}+\frac{1}{(1-x_j)^\a}\right]
    -\Delta t \eps C_\a h \sum^{k=J-j}_{k=-J-j,k\neq
0}\!\!\!\!{''} \;
    {\frac{1}{|x_k|^{1+\a}} }\geq 0.
\label{eq.cdm1}
\end{equation}

   Similarly, when the condition \eqref{eq.cdm1} is met,
we can also prove that $P_j^{n+1} \geq 0$ given $P_j^n \geq 0$ 
for all $j$, because all coefficients for the elements of $P^n$
in \eqref{eq.fe} are non-negative.
%Note we cannot prove the statement that $m \leq
%P_j^{n+1}$ if $m \leq P_j^n$.

  Next, we derive a sufficient condition \eqref{eq.cdm} that guarantees 
the inequality \eqref{eq.cdm1}.
We have the following estimate 
\begin{equation}
\begin{split}
& \Delta t \eps C_\a h \sum^{k=J-j}_{k=-J-j,k\neq 0}\!\!\!\!{''} \;
    {\frac{1}{|x_k|^{1+\a}} } +\frac{\Delta t \eps C_\a}{\a}
     \left[\frac{1}{(1+x_j)^\a}+\frac{1}{(1-x_j)^\a}\right]\\
     &= \Delta t \eps C_{\a} \left[h \sum^{k=J-j}_{k=-J-j,k\neq
0}\!\!\!\!{''} \;
    {\frac{1}{|x_k|^{1+\a}} } + \int_{-\infty}^{-1-jh}\frac{1}{|y|^{1+\a}}dy +
    \int_{1-jh}^{\infty}\frac{1}{|y|^{1+\a}}dy \right]\\
    &\leq \Delta t \eps C_{\a}
    \left[\frac{2h}{h^{1+\a}} + \int_{-1-jh+\frac{h}{2}}^{-h}\frac{1}{|y|^{1+\a}}dy +
     \int_{h}^{1-jh-\frac{h}{2}}\frac{1}{|y|^{1+\a}}dy +
 \int_{-\infty}^{-1-jh}\frac{1}{|y|^{1+\a}}dy +
    \int_{1-jh}^{\infty}\frac{1}{|y|^{1+\a}}dy \right]\\
    & \leq  \Delta t \eps C_{\a} \left[ \frac{2}{h^{\a}} + 2\int_{h}^{\infty} \frac{1}{y^{1+\a}}dy \right]
 = \Delta t \eps C_{\a} \left[\frac{2}{h^{\a}} \left(1+ \frac{1}{\a}\right) \right]
\end{split}
\label{eq.est}     
\end{equation}
where $x_j=jh$.

From the estimate \eqref{eq.est}, the inequality \eqref{eq.cdm1} holds
if 
\begin{equation}
    \Delta t \eps C_{\a}\left[\frac{2}{h^{\a}}
   \left(1+ \frac{1}{\a}\right)
  - \frac{2\zeta(\a -1)}{h^{\a}}\right] 
   = \frac{\Delta t}{h^\a} 2 \eps C_{\a}\left[ 1+ \frac{1}{\a} - \zeta(\a
-1) \right]\leq 1,
\end{equation}
which is equivalent to the condition \eqref{eq.cdm}.
\end{proof}

Next, we turn to the scheme for the natural condition.
\begin{proposition}\label{mpn}
When both $d$ and $f$ vanish, 
the scheme \eqref{eq.sdn} corresponding to the natural 
condition with forward Euler for time integration
satisfies the discrete maximum principle,
i.e., $0\leq P_j^0\leq M$ for all $j$'s implies $0\leq P_j^{n}\leq M$
for all  $n>0$ and $j$'s,
if the inequality \eqref{eq.cdm} holds.
\end{proposition}

\begin{proof}
With the explicit Euler method for time integration, Eq.~\eqref{eq.sdn}
becomes 
\begin{equation}
 P_j^{n+1}=P_j^{n} - \Delta t \eps C_\a \zeta(\alpha-1) h^{-\a}(P_{j-1}^n -2 P_j^n +P_{j+1}^n)
+  \Delta t \eps C_\a h \sum^{k=J-j}_{k=-J-j,k\neq 0}\!\!\!\!{''} \;
    {\frac{P_{j+k}^n -P_j^n }{|x_k|^{1+\a}} },
\label{eq.fen}
\end{equation}
where $J=L/h$ and $L\gg 1$. The difference between the discrete
equations corresponding to the absorbing condition and 
natural condition, \eqref{eq.fe} and \eqref{eq.fen}, 
is that  \eqref{eq.fen} does not have the term 
$\displaystyle \frac{\Delta t\eps C_\a}{\a}
\left[\frac{1}{(1+x_j)^\a}+\frac{1}{(1-x_j)^\a}\right]
 P_j^n$. 

Similar to the proof of Proposition~\ref{prop1}, \eqref{eq.fen}
satisfies discrete maximum principle provided that
\begin{equation} 
   1+ 2\frac{\Delta t \eps C_\a \zeta(\alpha-1)}{
h^{\a}}
    -\Delta t \eps C_\a h \sum^{k=J-j}_{k=-J-j,k\neq
0}\!\!\!\!{''} \;
    {\frac{1}{|x_k|^{1+\a}} }\geq 0.
\label{eq.cdm2}
\end{equation}
And, a similar estimate 
\begin{equation}
\begin{split}
&\Delta t \eps C_\a h \sum^{k=J-j}_{k=-J-j,k\neq 0}\!\!\!\!{''} \;
    {\frac{1}{|x_k|^{1+\a}} } 
    \leq \Delta t \eps C_{\a}
    \left[\frac{2h}{h^{1+\a}} + \int_{-L-jh+\frac{h}{2}}^{-h}\frac{1}{|y|^{1+\a}}dy +
     \int_{h}^{L-jh-\frac{h}{2}}\frac{1}{|y|^{1+\a}}dy
  \right]\\
     \leq & \Delta t \eps C_{\a} \left[ \frac{2}{h^{\a}} + 2\int_{h}^{\infty} \frac{1}{y^{1+\a}}dy \right]
 = \Delta t \eps C_{\a} \left[\frac{2}{h^{\a}} \left(1+ \frac{1}{\a}\right) \right]
\end{split}
\label{eq.est2}
\end{equation}
leads to the same condition \eqref{eq.cdm} for maximum principle. 
\end{proof}

Now, we are ready to show that both semi-discrete schemes \eqref{nm1D3} 
and \eqref{eq.sdn} with the third-order TVD Runge-Kutta 
time discretization \eqref{eq.rk3} satisfy 
the discrete maximum principle. 
\begin{proposition}\label{mprk}
When both $d$ and $f$ vanish, 
the semi-discrete schemes \eqref{nm1D3} and \eqref{eq.sdn}, 
corresponding to the absorbing and the natural 
conditions respectively, with the third-order Runge-Kutta 
time discretization \eqref{eq.rk3} satisfy the discrete maximum principle,
i.e., $0\leq P_j^0\leq M$ for all $j$'s implies $0\leq P_j^{n}\leq M$
for all  $n>0$ and $j$'s,
if the inequality \eqref{eq.cdm} holds.
\end{proposition}
\begin{proof}
Propositions \ref{prop1} and \ref{mpn} have shown that 
the semi-discrete schemes \eqref{nm1D3} and \eqref{eq.sdn}
with the forward Euler time discretization
satisfies the discrete maximum principle 
when the inequality \eqref{eq.cdm} holds.
Based on the proof of Proposition~2.1 in \cite{shu1988efficient}, 
the explicit Runge-Kutta method \eqref{eq.rk3} satisfies 
the maximum principle under the same restriction \eqref{eq.cdm},
because it can be written 
as a convex combination of the operations that satisfy
the maximum principle due to Propositions \ref{prop1} and \ref{mpn}. 
\end{proof}

\begin{figure}[h]
\begin{center}
\includegraphics*[width=12cm]{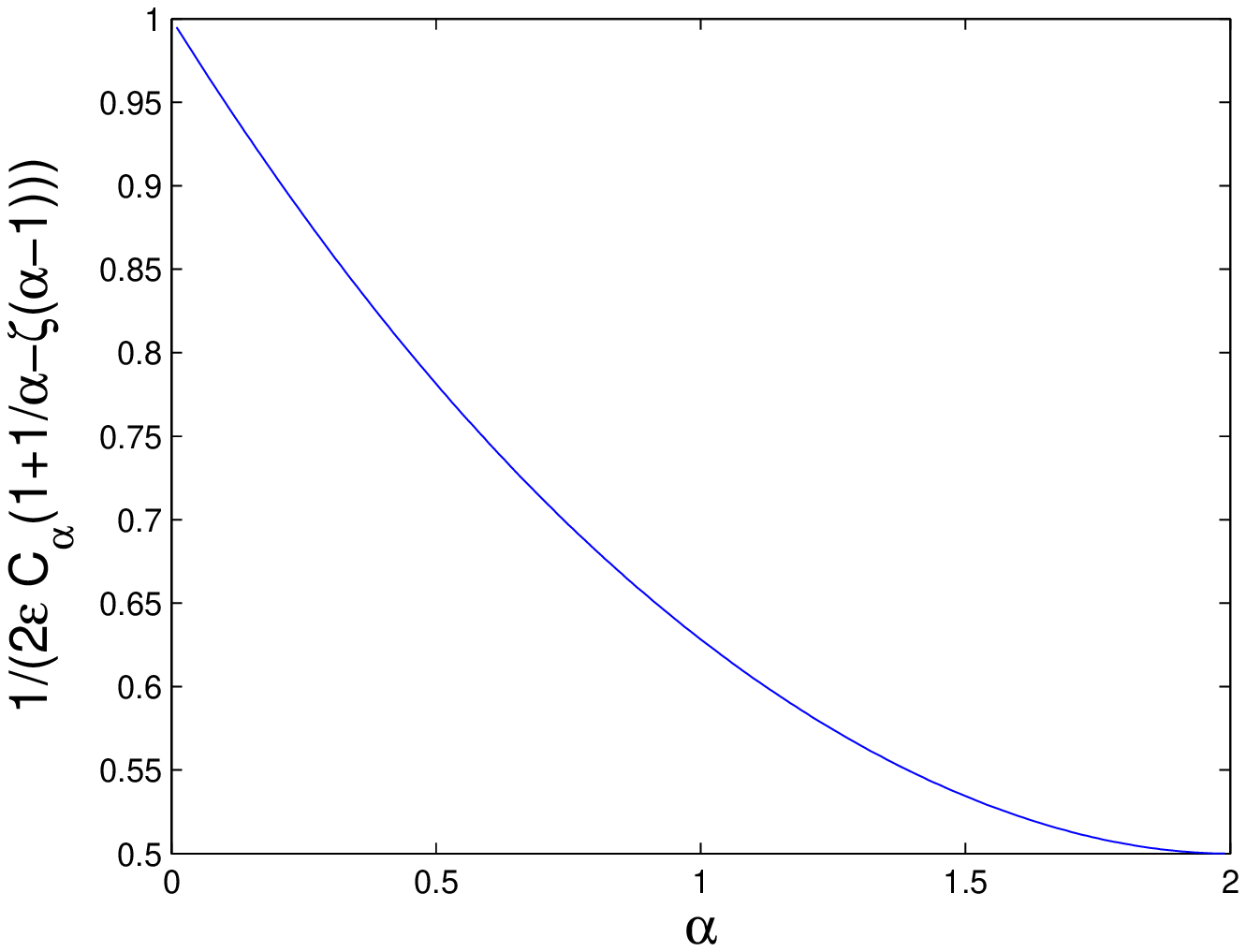}
\end{center}
\caption{The dependence of the threshold $\frac{1}{2 C_{\a} \left[1+
\frac{1}{\a} - \zeta(\a -1) \right]}$ on $\alpha$ with $\eps=1$.}
{\label{threshold}}
\end{figure}
\begin{remark}\label{sim}
It is interesting to note that the condition for the explicit method
corresponding to $\a$-stable 
process \eqref{eq.cdm} limits the ratio of the time step $\Delta t$
to the spatial grid size $h^\a$, while the corresponding condition 
for the Gaussian diffusion term $p_{xx}$ 
limits the ratio $\Delta t/h^2$. Further, the $\a$-stable process
is Gaussian when $\a=2$. The upper bound of the ratio for satisfying 
the maximum principle, i.e., the value of the constant on the right-hand
side of the inequality \eqref{eq.cdm}, is plotted for $0<\a<2$ 
in Fig.~\ref{threshold}. Surprisingly, the value of the threshold
decays from $1$ to $1/2$ as $\a$ increases from $0$ to $2$. 
To satisfy the maximum principle, we recall
the condition for the classic explicit scheme corresponding 
to the heat equation $p_t=\eps p_{xx}$, 
the forward Euler in time and the second-order central differencing in space,
is $\Delta t/h^2< 1/2$, matching the limit
of the inequality of \eqref{eq.cdm} as $\a$ approaches $2$. 
This agreement suggests that the condition \eqref{eq.cdm}
might be sharp for satisfying the maximum principle.
\end{remark}

\subsection{Stability and Convergence}
Due to the linearity of the schemes \eqref{eq.fe} and \eqref{eq.fen},
their stability can be deduced by means of the maximum principle shown
in the previous section \S~\ref{sec.mp}.
Therefore, the explicit schemes \eqref{eq.fe} and \eqref{eq.fen}
are stable when the ratio $\Delta t/h^\a$ satisfies
the condition \eqref{eq.cdm}.

Now we explore the convergence condition of the forward Euler 
scheme \eqref{eq.fen} for the equation $(\ref{FPE555})$ with 
the natural far-field condition. Although our proof
is for the case where the diffusion $d$ and 
the drift $f$ are absent, the extension to the cases in which both
$d$ and $f$ are present is straightforward. Also, one can easily extend
the following proof of 
convergence result to other different time integration schemes 
for the semi-discrete equations \eqref{nm1D3} and \eqref{eq.sdn}.

\begin{proposition}
 The solution $U_j^n$ to the numerical scheme \eqref{eq.fen} for the natural
far-field condition converges to the analytic solution to \eqref{FPE555}
for $x_j$ in $[-L/2,L/2]$ when the refinement path satisfies \eqref{eq.cdm}
and the length of the integration interval $2L$ in \eqref{eq.fen} 
tends to $\infty$. 
\end{proposition}

\begin{proof}
Denote the discretization error by $e_j^n:=P_j^n - p(x_j,t_n)$, 
where $p$ is the analytic solution. Defining the truncation error
$T_j^n$ as the difference between the left-hand and right-hand sides of
\eqref{eq.fen} evaluated at $(x_j,t_n)$ 
when the numerical solution $P_j^n$ is replaced
with the analytic solution $p(x_j,t_n)$ divided by $\Delta t$, 
the linearity of the equation leads to 
\begin{equation}
%  \begin{split}
 e_j^{n+1}=e_j^{n} - \Delta t \eps C_\a \zeta(\alpha-1) h^{2-\a}\frac{e_{j-1}^n -2 e_j^n +e_{j+1}^n}{h^2} 
   +  \Delta t \eps C_\a h \sum^{k=J-j}_{k=-J-j,k\neq
0}\!\!\!\!\!{''} \;
    {\frac{e_{j+k}^n -e_j^n }{|x_k|^{1+\a}} } - \Delta t T_j^n, 
%  \end{split}
\label{eq.defen}
\end{equation}
where $J=L/h$.
The leading-order terms in the truncation error are given by 
\begin{equation}
\begin{split}
T_j^n = & \frac{1}{2} p_{tt}(x_j,t_n) \Delta t 
   + \frac{1}{12}\eps C_\a\zeta(\a-1) p_{xxxx}(x_j,t_n) h^{4-\a} 
   - C \eps \frac{\partial}{\partial y} 
       \left. \left( \frac{p(x_j+y,t_n)}{|y|^{1+\a}} \right) \right|_{y=-L-x_j}^{y=L-x_j}
    h^2 \\ 
 & - D \eps\zeta(\a-3) p_{xxxx}(x_j,t_n) h^{4-\a} 
   + \int_{\{-\infty,-L-x_j\} \cup \{L-x_j, \infty\} } 
       \frac{p(x_j+y,t)-p(x_j,t)}{|y|^{1+\a}} \,{\rm d}y 
   + \cdots,
\end{split}
\end{equation}
where $C$ and $D$ are constants.  Thus, the truncation error is uniformly bounded,
\begin{equation}
|T_j^n| \leq \bar{T} := O(\Delta t) + O(h^2) + O(L^{-\a}).
\end{equation}
Using the same argument as in the proof in Proposition~\ref{mpn}, 
under the condition \eqref{eq.cdm}, the coefficients of $e_m^n$ are nonnegative
for all $m$'s. Denoting $E^n= \max_{j} |e_j^n|$, we have
\begin{equation}
E^{n+1} \leq E^n + \Delta t \bar{T}, \quad \text{then} \quad
E^n \leq E^0 + n\Delta t \bar{T} = n\Delta t \bar{T}.
\end{equation}
Therefore, the convergence is proved.

%\begin{equation}
%  \begin{split}
% e_j^{n+1}=&e_j^{n} - \Delta t \eps C_\a \zeta(\alpha-1) h^{2-\a}\frac{e_{j-1}^n -2 e_j^n +e_{j+1}^n}{h^2} 
%  -\frac{\Delta t\eps C_\a}{\a} 
%    \left[\frac{1}{(1+x_j)^\a}+\frac{1}{(1-x_j)^\a}\right] e_j^n \\
%  &  +  \Delta t \eps C_\a h \sum^{k=J-j}_{k=-J-j,k\neq
%0}\!\!\!\!{''} \;
%    {\frac{e_{j+k}^n -e_j^n }{|x_k|^{1+\a}} } - \Delta t T_j^n. 
%  \end{split}
%\label{eq.defe}
%\end{equation}
%From \cite{Ting12} or Eq.~\eqref{eq.oc}, the leading-order truncation error
%is given by 
%\begin{equation}
%\begin{split}
%T_j^n = & \frac{1}{2} p_{tt}(x_j,t_n) \Delta t 
%   + \frac{1}{12}\eps C_\a\zeta(\a-1) p_{xxxx}(x_j,t_n) h^{4-\a} \\
%   & - C_1\eps \left[ \frac{p_x(1,t_n)}{|1-x_j|^{1+\a}} 
%   - \frac{p_x(-1,t_n)}{|1+x_j|^{1+\a}} \right] h^2 
%   - D_2 \eps\zeta(\a-3) p_{xxxx}(x_j,t_n) h^{4-\a} + \cdots,
%\end{split}
%\end{equation}
%where $C_1$ and $D_2$ are constants. 
\end{proof}

\begin{remark}
Because the derivatives of the density function $p$ might be infinite
at the boundaries of the domain $(-1,1)$ for the absorbing condition,
as shown in next section,
the truncation error of the modified trapezoidal rule is
not uniformly second-order. The numerical analysis of the convergence
for \eqref{eq.fe} will be studied in future. 
\end{remark}

%\subsection{positive preserving property}
%(ref: maximum principle satisfying and positivity preserving high
%order schemes for conservation laws by Xiangxiong Zhang and Chiwang
%Shu.)
%
%\subsection{TVD Condition}
%{\bf There are conditions for the second equation in splitting
%method. Generally, how to derive? }

\section{Numerical results}
\label{sec.nr}

%\begin{figure}[]
%\begin{center}
%\includegraphics*[width=6.5cm,height=5cm]{./figure/exact_Burrage_alpha.eps}
%\includegraphics*[width=6.5cm,height=5cm]{./figure/exact_Chapter3.eps}
%\end{center}
%\caption{$d=0,f=0,\varepsilon =1$. left:Burrage, right: series
%representation from Chapter3 }
%\end{figure}
%
%
%
%
%\begin{figure}[]
%\begin{center}
%\includegraphics*[width=6.5cm,height=5cm]{./figure/exact_Burrage.eps}
%\includegraphics*[width=6.5cm,height=5cm]{./figure/plotdiffT.eps}
%\end{center}
%\caption{$d=0,f=0,\varepsilon =1$. left:Burrage's exact solution,
%right: our solution}
%\end{figure}

\subsection{Verification}

\begin{figure}[h]
\begin{center}
\includegraphics*[width=7.8cm]{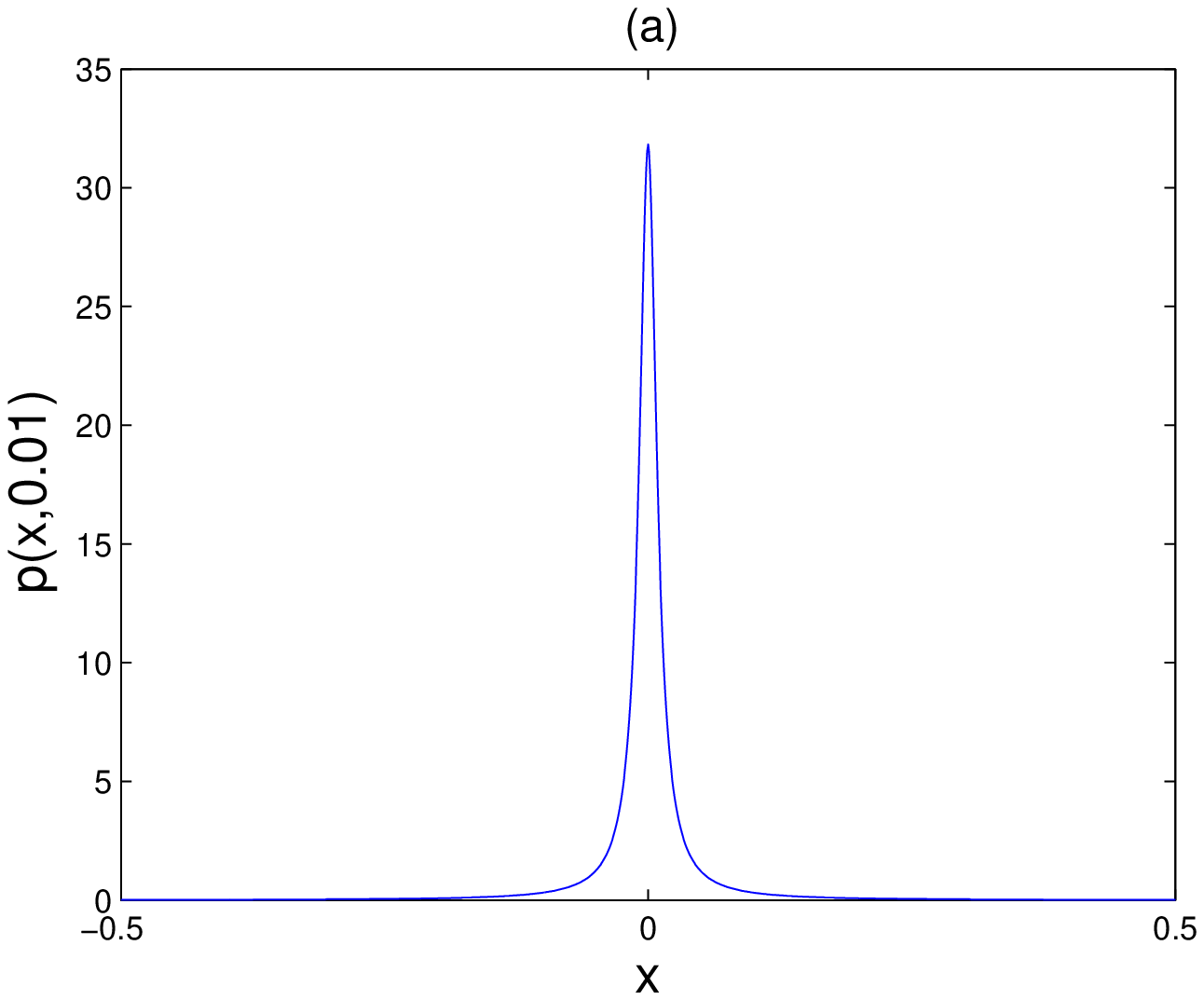}
\includegraphics*[width=7.8cm]{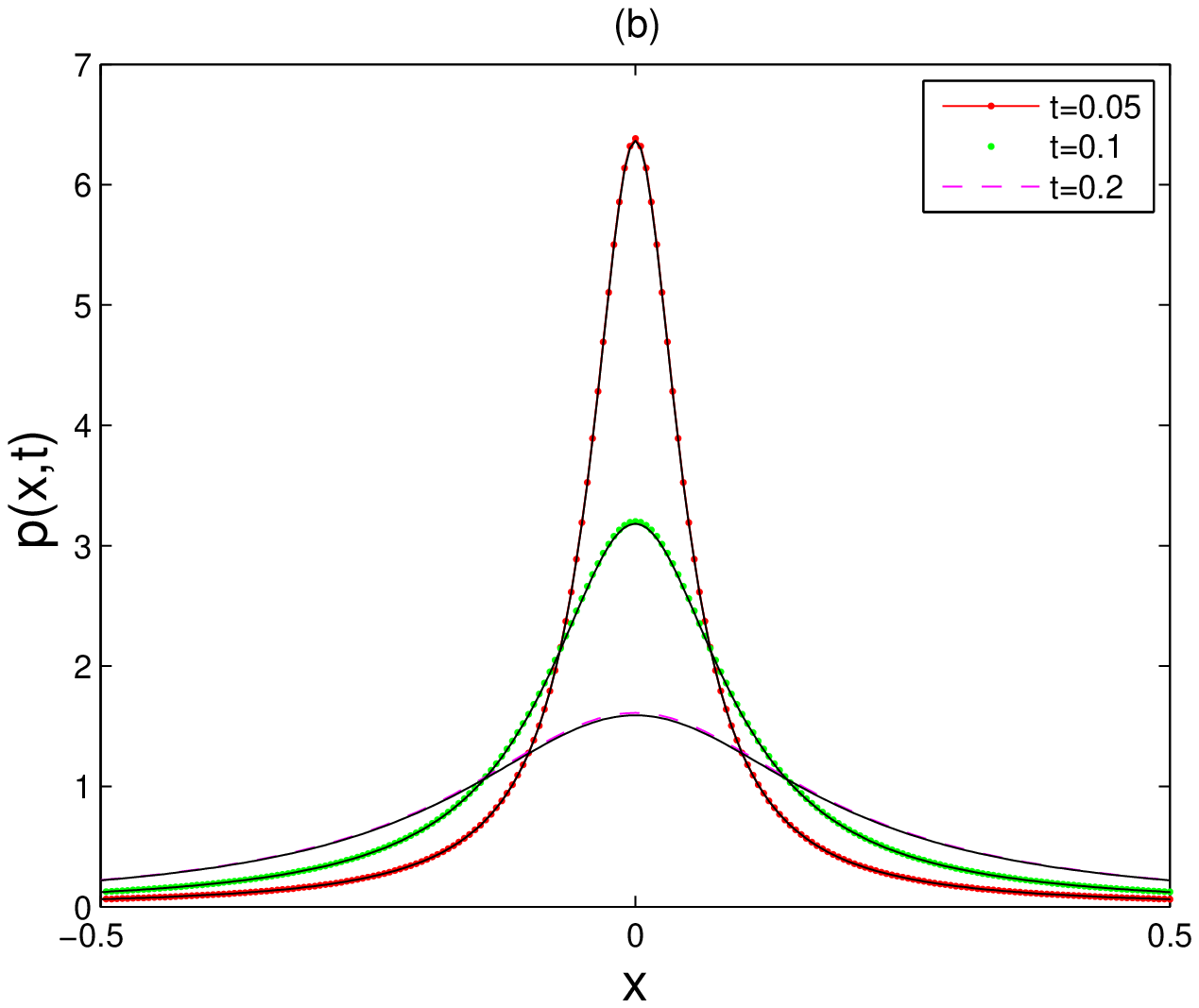}
\end{center}
\caption{The left plot is the initial condition
$p(x,0.01)=\frac{0.01}{\pi(0.01^2+x^2)}$. The right graph plots
both the numerical solution and the exact solution at times $t=0.05, 0.1, 0.2$,
where the numerical solutions are drawn by dotted or dashed lines 
and the exact solutions are drawn by solid lines.} 
\label{comp_cauchy}
\end{figure}
For some special cases, the analytic or exact solutions can be
obtained by inspecting the characteristic functions in the form of
Fourier Transform (\cite{imbert2005non}) or series expansion
(\cite{bergstrom1952some}). We first compare our numerical solution with the
exact solution $\displaystyle{p(x,t)=\frac{t}{\pi(t^2+x^2)}}$  for
the Cauchy case with the natural far-field condition, which
corresponds to the FP equation~\eqref{ffp4} with $d=0$, $f=0$,
$\varepsilon=1$, $\alpha=1$. To avoid numerical complication of a
delta function, we start our numerical computation from time
$t=0.01$ by setting the initial condition to be
$\displaystyle{p(x,0.01)=\frac{0.01}{\pi(0.01^2+x^2)}}$ and
approximating the natural far-field condition by truncating the
computational domain to $(-50,50)$ using the numerical scheme  
given in Eq.~(\ref{eq.sdn}).
Both the exact and the numerical solutions are shown in
Fig.~\ref{comp_cauchy} at the times $t=0.05, 0.01$ and $0.2$. In the
numerical simulations, we have chosen the spatial resolution
$h=0.001$ and the time step size $\Delta t= 0.5 h$. As the graphs
show, the numerical results are indistinguishable from the exact
solution, implying our numerical methods produce correct and
accurate results.

\begin{figure}[h]
\begin{center}
\includegraphics*[width=9.5cm]{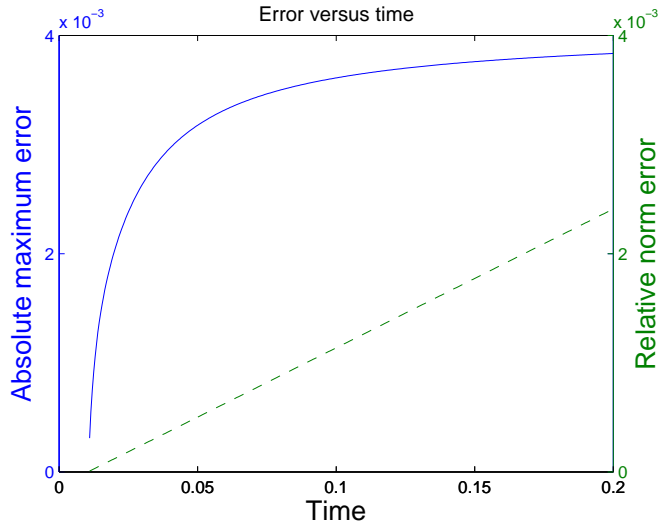}
\end{center}
\caption{The growth of the absolute error in maximum norm and the
relative error in 2-norm in time for the numerical solution With the
computational domain size $(-50,50)$, $h=0.001$ and $\Delta t=0.5h$.
}
%{With domain size $(-50,50)$, $h=0.005$ and $\Delta
%t=0.5h^2$, this plot shows the absolute maximum error and relative
%norm error of numerical solution and true solution as time goes from
%$0.011$ to $0.1$. (\textcolor{red}{lower: $h=0.001$, $t=0.5h$,
%choose which one?} ) }
{\label{error_time}}
\end{figure}
 We also examine the maximum absolute error and the relative error in 2-norm
of our numerical solution, defined by $\|\eb_n\|_{\infty} \equiv
\max_j |P^n_j - p(x_j,t_n)|$ and $\|\eb_n \|_2/\|(p(x_j,t_n))\|_2$
respectively. Here, the error vector $\eb_n$ is defined as $\eb_n :=
(P^n_j- p(x_j,t_n))$, the difference between the numerical solution
$(P^n_j)$ and the exact solution $(p(x_j,t_n))$ at the time $t_n$.
Figure~\ref{error_time} shows the sizes of the errors of our
numerical solutions obtained by using the computational domain
$(-50,50)$, $h=0.001$ and $\Delta t=0.5h$, as the time increases
from $0.01$ to $0.2$. The absolute error increases initially but
levels off quickly as the solution itself becomes small, while the
relative error grows linearly in time. Up to time $t=0.2$, the
relative error is small (less than 0.3\%).

%{\bf Question 1:} Check whether splitting method and FTCS method
%can get the same result?\\
%{\bf I think the error comes from the second equation in the
%splitting method. Namely, do we need to use higher order WENO method
%instead of MC limiter? } In fact,  the error basically comes from
%the peak point around $x=0$.
%
%\textcolor{blue}{As I said above, let's use WENO for the 1st-order derivative.}
%
%%
%%\begin{figure}[h]
%%\begin{center}
%%\includegraphics*[width=6.5cm,height=5cm]{./figures/compare.eps}
%%\includegraphics*[width=6.5cm,height=5cm]{./figures/compare1.eps}
%%\end{center}
%%\caption{$d=0,f=0,\varepsilon =1, \alpha=1.5$}
%%\end{figure}
%
%\begin{figure}[h]
%\begin{center}
%%\includegraphics*[width=6.5cm,height=5cm]{./figures/relative_comp1.eps}
%\includegraphics*[width=6.5cm,height=5cm]{./figures/relative_comp2.eps}
%\includegraphics*[width=6.5cm,height=5cm]{./figures/relative_comp5.eps}
%\includegraphics*[width=6.5cm,height=5cm]{./figures/relative_comp3.eps}
%\includegraphics*[width=6.5cm,height=5cm]{./figures/relative_comp4.eps}
%
%\end{center}
%\caption{The comparison between splitting method and FTCS with
%$d=0,f=0,\varepsilon =1, \alpha=1.5$. The upper two are from
%splitting method with MC limiter and lower two are from splitting
%method with second equation uncounted (basically extrapolation). The
%left two are the relative error at $x=0.8$ and the right two are the
%relative error in norm sense, i.e. $norm(y1-y2)/norm(y1),$ where
%$y1$ and $y2$ are solutions from two methods respectively at
%different time $T$.}
%\end{figure}

Next, we check the order of convergence of our numerical methods in
space. In theory \cite{sidi88}, with the correction term $-\eps
C_{\alpha}\zeta(\alpha-1)h^{2-\alpha}$, the error between the
integro-differential term $\eps
k_\alpha(-\Delta)^{\frac{\alpha}{2}}$ in Eq.~\eqref{ffp4} and its
numerical approximation in Eq.~\eqref{nm1D3} is given by
\begin{equation}
 %\begin{split}&
   -\sum\limits_{\mu=1}^{m-1}C_{\mu} \left[
   \frac{\partial^{2\mu-1}}{\partial y^{2\mu-1}}
  \left(\left. \frac{p(x+y,t)}{|y|^{1+\a}}\right)\right|_{y=-L}^{y=L}
  \right] h^{2\mu}
 %\\
%&
 - \sum\limits_{\mu=2}^{m-1}D_{\mu} \zeta(\alpha+1-2\mu)
\frac{\partial^{2\mu} p}{\partial x^{2\mu}} (x,t) \,
h^{2\mu-\alpha}+O(h^{2m}),
%\\
% \end{split}
\label{eq.oc}
\end{equation}
where $L\gg |x|$, and $C_{\mu}$ and $D_{\mu}$ are constants related
to $\mu$.

\begin{table}[h]
\begin{tabular}{ |c | c | c | c | c | c| c| c| }
\hline
$h$                &  $0.1$  &  $0.1/2$ &   $0.1/2^2 $  &  $0.1/2^3 $ & $0.1/2^4$ & $0.1/2^5$& $0.1/2^6$ \\
\hline Error & $1.01 $ &  $0.264$  &  $0.0554$  &
$4.23\times10^{-3}$ &
$8.41\times10^{-5}$ &  $4.13\times10^{-5}$  &  $3.93\times10^{-5}$\\
\hline Order & $1.94  $ & $ 2.25  $ & $ 3.71  $ & $5.65 $ &
$1.03$ &$0.0734$ &\\
 \hline
\end{tabular}
\caption{The error of the numerical solution of $p(x,t)$ at the
fixed point $(x,t)=(0.1,0.02)$ and the numerical order of
convergence for difference sizes of the resolution $h$. The
numerical solutions are obtained without Richardson extrapolation.}
\label{tab.1}
\end{table}
In Table~\ref{tab.1}, we provide the computed order of convergence
in space using different sizes of $h$ for the Cauchy case $\a=1$
with the natural far-field condition, no deterministic driving force
$f=0$ and no Gaussian diffusion $d=0$. The error in the numerical
solution is computed with the computational domain
$(-L,L)=(-100,100)$ at a fixed point $(x,t)=(0.1,0.02)$ by comparing
with the known exact solution $p(0.1,0.02)$ and, again, the initial
data is taken from the exact solution at time $t=0.01$. The
numerical order of convergence is calculated from the formula
$\log_2 |e(h)/e(h/2)|$, where $e(h)$ denotes the error at the point
$(x,t)=(0.1,0.02)$ with the spatial resolution $h$.

The results in Table~\ref{tab.1} shows that, for relatively large or
moderate sizes of $h$, the order of convergence is either close to
or better than two. Second-order convergence is expected from the
theory \eqref{eq.oc}. The higher orders of convergence shown in
Table~\ref{tab.1} can be explained as follows. First, for
$\alpha=1$, the second summation in the quadrature error
\eqref{eq.oc} vanishes due to $\zeta(-2)= \zeta(-4)= \cdots =0$.
Second, the coefficients in the terms of the first summation in
\eqref{eq.oc} are small because the unknown function $p$ and its
spatial derivatives vanish at the infinity and are small when $L$ is
large. However, the error in the numerical solution ceases to
decrease further when $h$ is small, since the size of the error is
comparable to the error caused by truncating the infinite domain
$(-\infty,\infty)$ of the original problem to the finite-sized
computational domain $(-L,L)$.

\begin{table}[h]
\begin{tabular}{ |c | c | c | c | c | c| c| c| }
\hline
$h$                 &  $0.1$  &  $0.1/2$ &   $0.1/2^2 $  &  $0.1/2^3 $ & $0.1/2^4$ & $0.1/2^5$& $0.1/2^6$ \\
\hline Error & $0.379 $ &  $0.0989$ & $ 0.0208 $ &
$1.58\times10^{-3}$ & $1.69\times10^{-5} $ & $8.78\times10^{-7} $  &
$ 1.10\times10^{-7}
$\\
\hline Order & $1.94$  &   $2.25$ &  $3.72$ & $6.54$
 &  $4.29$  &  $3.00 $&\\
 \hline
 %$\mathrm{Order}^2$ &
%$1.9356  $ & $ 2.2513  $ & $ 3.7099  $ &
%$5.6539 $ & $6.5078$&  $ 3.0030$&\\
%\hline
\end{tabular}
\caption{The error of the numerical solution of $p(x,t)$ at the
fixed point $(x,t)=(0.1,0.02)$ and the numerical order of
convergence for difference sizes of the resolution $h$. The
numerical solutions are obtained with Richardson extrapolation.}
\label{tab.2}
\end{table}
Next, we perform Richarsdon extrapolation with respect to $L$, half
of the size of the computational domain for computing the numerical
solution of $p(0.1,0.02)$, reducing this truncation error small
enough such that the total error is dominated by the quadrature
error \eqref{eq.oc}. The extrapolation is given by $\displaystyle
\frac{1}{3}P(L) - 2 P(2L) + \frac{8}{3} P(4L) = p(0.1,0.02) +
O(\frac{1}{L^3})$, where we denote the numerical solution of
$p(0.1,0.02)$ with the computational domain $(-L,L)$ by $P(L)$.
Table~\ref{tab.2} shows that the numerical results using the
Richardson extrapolation continue to improve as $h$ decreases and
the order of convergence is better than second-order.

% At $x=0$, the order of convergence is
%$2.7495, 1.8321,
% 1.2319$ when final time $t= 0.2, 0.5, 1$ respectively.\\
%$ t=0.01$ to $0.02$, when $h=1/50, 1/100, 1/200, 1/400, 1/800$,
%order will be $2.62, 3.01, 1.28, 0.27$.  \textcolor{red}{Delete
%this??} \textcolor{blue}{for $h=1/10, 1/20, 1/40, 1/80 ,1/160,
%1/320$, $dt = 10^{-4}$ on $domain(-50,50)$, the order is still
%not good : \\\\
%$1.6928,   2.5855,    2.6271,    1.6687,  0.4704, $if �� $T0=0.05$,
%$T=0.06$.\\\\
% $2.6449,    3.2475 ,   1.7668,    0.4862,    0.0743,$ if
%$T0=0.05, T=0.1$.\\\\
%$ 3.0170 ,   1.2849,    0.2706  ,  0.0378,    0.0048 $ if
%$T0=0.1,T=0.2$\\\\
%$ 0.4512  ,   0.9954 ,   2.1595 ,   3.4591  ,  2.9367,  1.8517,$ if$
%T0=0.01,T=0.02$ (use $h=1/10, 1/20, 1/40, 1/80 ,1/160,
%1/320,1/640$).}\\

 %We choose a case where
% $\alpha=1.5, \varepsilon=1, d=1, f(x)=-x, r=1$ by {\bf fpe2.m}\\
%order$=   2.0650$. ($J=40, 80, 160, dt= 10^{-5}, T=0.01$)\\
%order$=    1.6537$. ($J=40, 80, 160, dt= 10^{-5}, T=0.1$)\\
% order$=  0.9198$. ($J=40, 80, 160, dt= 10^{-5}, T=1$)\\\\

\begin{figure}[h]
\begin{center}
\includegraphics*[width=12cm]{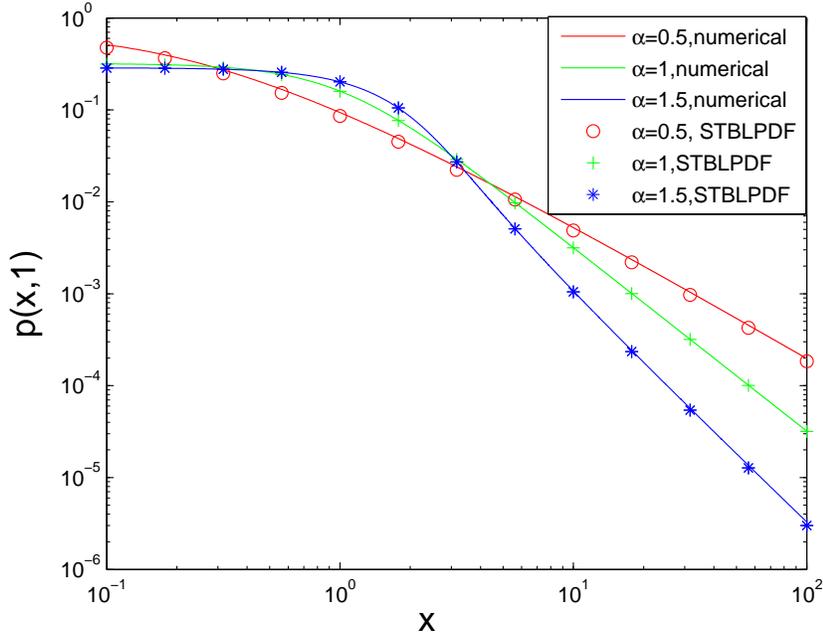}
\end{center}
\caption{Our numerical solutions (the solid lines) at $t=1$ with
$d=0,f=0,\varepsilon =1$ compared with the probability density
functions (the $o$'s, $+$'s, $*$'s symbols) of $\alpha-$stable
symmetric random variables. Our numerical results are computed with
the computational domain $(-110,110)$ with $h=0.001$ and $\Delta t=
0.5h^{\alpha}$, while the probability densities of the random
variables are obtained by the Matlab function
$\mathrm{stblpdf}.\mathrm{m}$ The Matlab package on $\alpha-$stable
distributions is downloaded from
$math.bu.edu/people/mveillet/html/alphastablepub.html$. They
simulated $\alpha-$stable random variables (\cite{chambers1976method}) and
numerically computed $\alpha-$stable distribution functions based on
numerical method in \cite{nolan1997numerical}.} 
\label{fig.comp_ana}
\end{figure}
As another verification, Fig.~\ref{fig.comp_ana} compares our
numerical solutions for different values of $\a=0.5, 1, 1.5$ at time
$t=1$ with the probability densities of the corresponding
$\alpha-$stable random variables. From the scaling property of
L\'evy process, we know that the solution to FP equation at time
$t=1$ is the same as the probability density of symmetric
$\alpha-$stable random variable. The heavy-tail property is easy to
see by looking at the power-law behavior for large $x$ in
Fig.~\ref{fig.comp_ana}. We start our computations from time
$t=0.01$ to avoid the singular behavior of the delta function, with
the initial profile provided by the exact solution to the Cauchy
case
$P(x,0.01)=\frac{t}{\pi(t^2+x^2)}=\frac{0.01}{\pi(0.01^2+x^2)}$.
Note that it induces a small error for the cases for which $\a \neq
1$. Nevertheless, our numerical solutions at time $t=1$ show that,
for large values of $x$, the slopes are equal to $-(1+\alpha)$,
implying the expected power-law behavior.

%For absorbing boundary condition, we use
%$$p_D(t,x,y) \asymp (1\wedge \frac{\delta_D(x)^{\alpha/2}}{\sqrt{t}})
%(1\wedge \frac{\delta_D(y)^{\alpha/2}}{\sqrt{t}})
%(t^{-d/\alpha}\wedge \frac{t}{|x-y|^{d+\alpha}})$$ where
%$\alpha=1.5$, $t=1$, $D=(-3,3)$ and
%\begin{equation*}
%p_D(t,x,y)\asymp
%\begin{cases}
%\frac{(3-y)^{\alpha/2}}{|y|^{1+\alpha}}, \ \ \ |y|\in (2,3)\\
%\frac{1}{|y|^{1+\alpha}},\ \ \ |y|\in(1,2)\\
%t^{-d/\alpha}\wedge \frac{t}{|x-y|^{d+\alpha}}, \ \ \ |y|\in(0,1)
%\end{cases}
%\end{equation*}
%\begin{figure}[]
%\begin{center}
%\includegraphics*[width=6.5cm,height=5cm]{./figures/natural_comp.eps}
%\includegraphics*[width=6.5cm,height=5cm]{./figures/absorbing_comp.eps}
%\end{center}
%\caption{$d=0,f=0,\varepsilon =1, \alpha=1.5$.left: natural boundary
%conditions and right: absorbing boundary conditions. red line:
%asymptotic solution given in Chen's paper. blue and yellow curves
%are from our numerical results. } \label{fig.comp_ana}
%\end{figure}

\subsection{Absorbing condition}

\begin{figure}
\begin{center}
\includegraphics*[width=11.8cm]{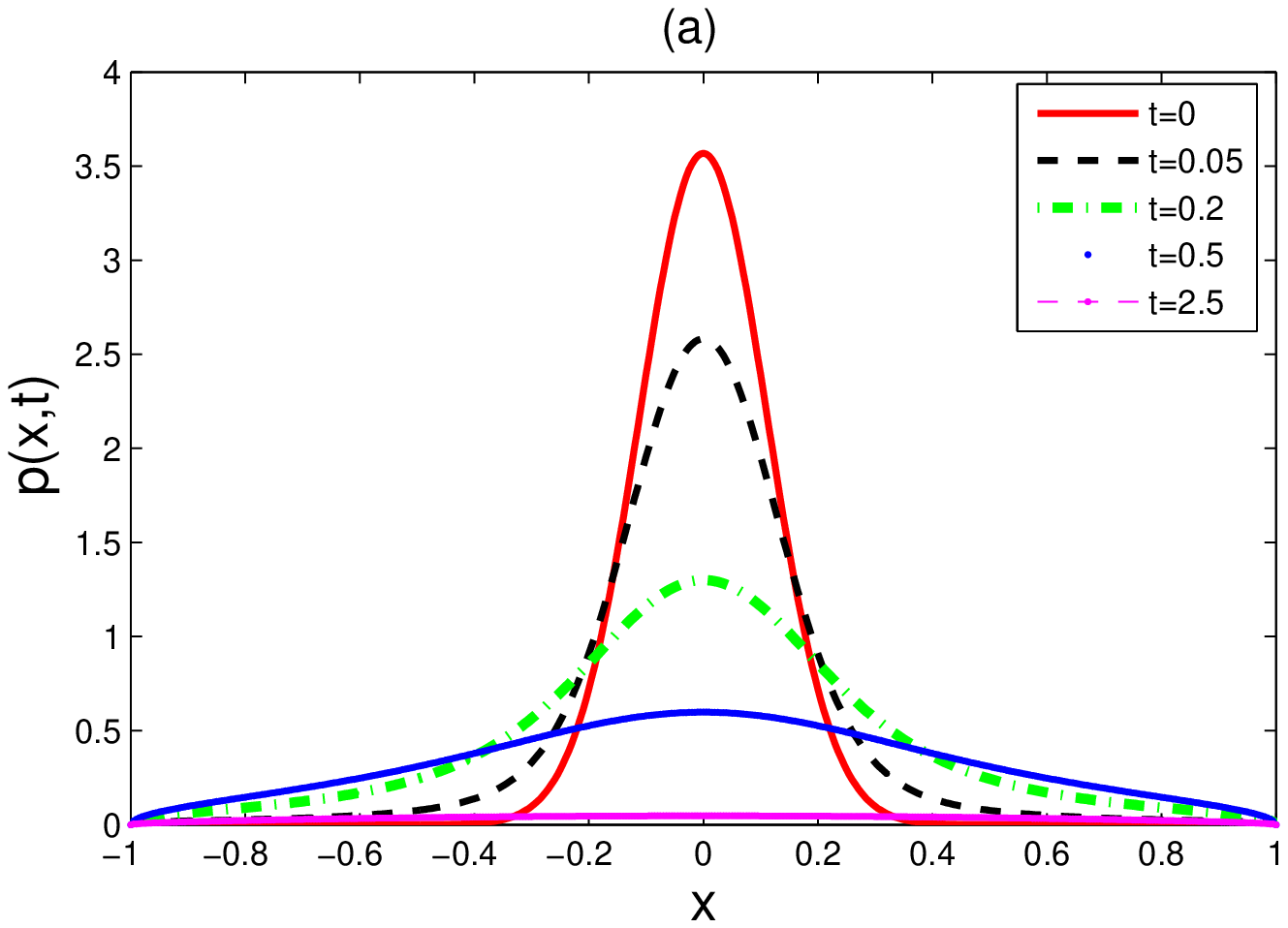}
\includegraphics*[width=11.8cm]{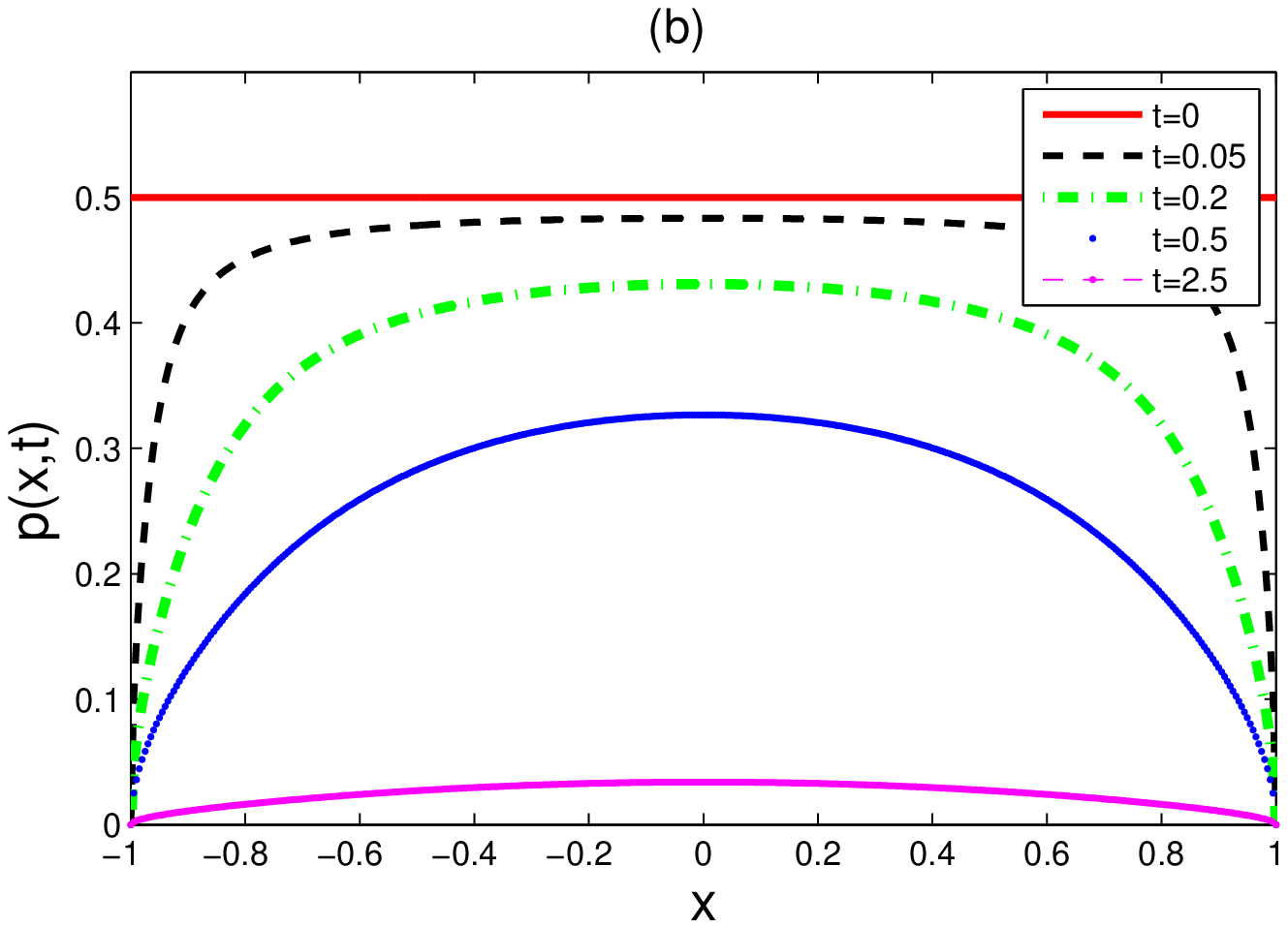}
\end{center}
\caption{The probability density function $p$ for the symmetric
$\a$-stable L\'evy process $X_t= {\rm d}L_t$ without any drift or
Gaussian diffusion ($d=0$) at different times $t=0, 0.05, 0.2, 0.5,
2.5$ with $f=0,\eps=1,\alpha=1$, $h=0.001$, $\Delta t=0.5h$ and the
absorbing condition $p(x,t)=0$ for $x \notin (-1,1)$. (a)
The initial data is a Gaussian distribution $p(x,0)=
\sqrt{\frac{40}{\pi}} e^{-\frac{x^2}{40}}$; (b) the initial profile
is the uniform distribution. } \label{time_sequence}
\end{figure}

In this section, we investigate the probability density $p(x,t)$
that a "particle" $X_t$ is located at the position $x$ at time $t$
given the probability profile of its initial position is Gaussian
($\displaystyle p(x,0)= \sqrt{\frac{40}{\pi}} e^{-\frac{x^2}{40}}$)
or uniform ($\displaystyle p(x,0)=\dfrac{1}{2} I_{\{|x|<1\}}$). The
absorbing condition implies the "particle" will vanish once
it is out of a specified domain. Figure~\ref{time_sequence} shows a
time sequence of the probability densities for the symmetric
$\a$-stable L\'evy process without any deterministic driving force
$f=0$ or any Gaussian diffusion $d=0$ but with $\a=1$ and $\eps=1$
at times $t=0, 0.05, 0.2, 0.5, 2.5$. As shown in
Fig.~\ref{time_sequence}(a), for an initial Gaussian distribution,
the density profiles keep the bell-type shape for small times while
the probabilities near the center decreases faster than those away
further from the center. For an initial uniform distribution,
Fig.~\ref{time_sequence}(b) demonstrates the density profile becomes
a smooth function after $t>0$, indicating the $\a$-stable process
has the effect of "diffusion". In this case, the maximum of the
density function stays at the center of the domain, because the
center is the furthest away from the boundary of the domain. In all
cases, the "particle" eventually escapes the bounded domain and
vanishes afterwards.

\subsubsection{Effect of different jump measures}

\begin{figure}
\begin{center}
\includegraphics*[width=7.9cm]{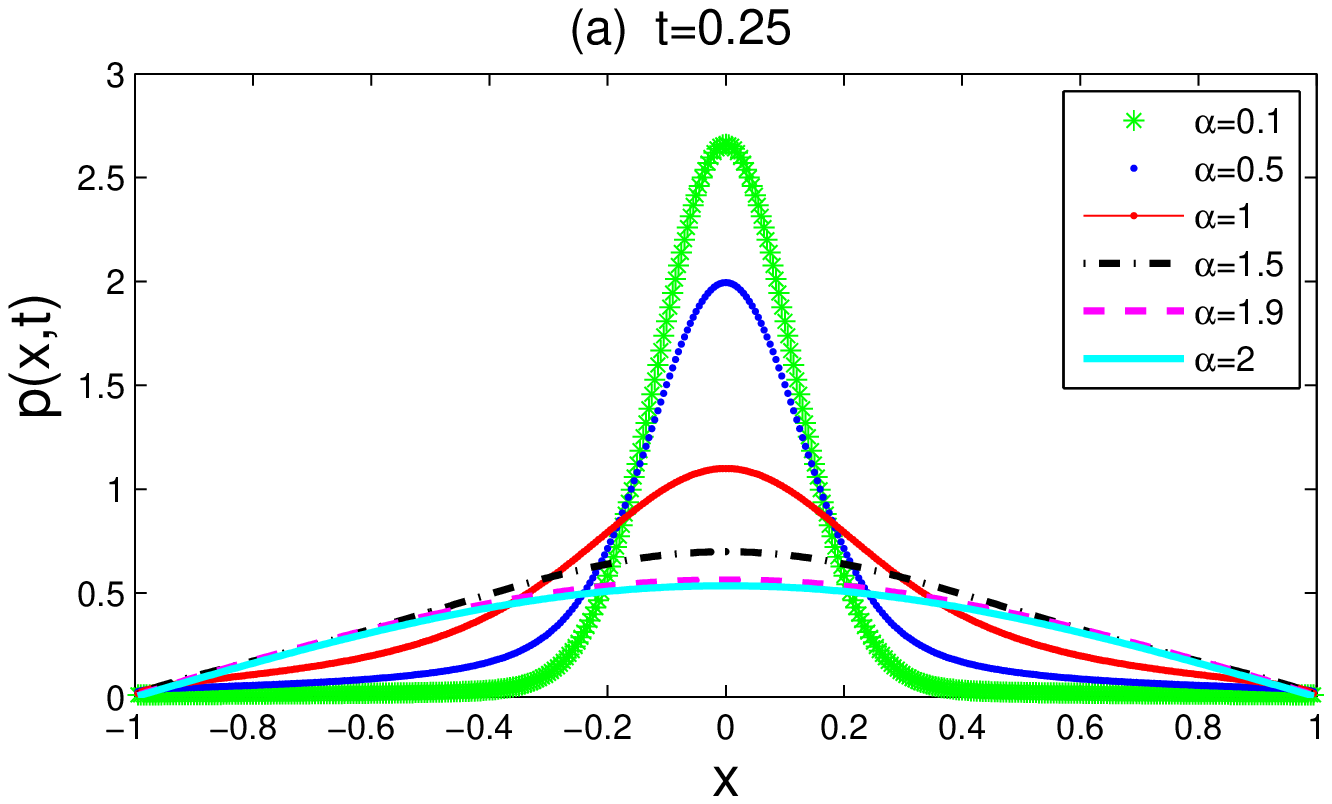}
\includegraphics*[width=7.9cm]{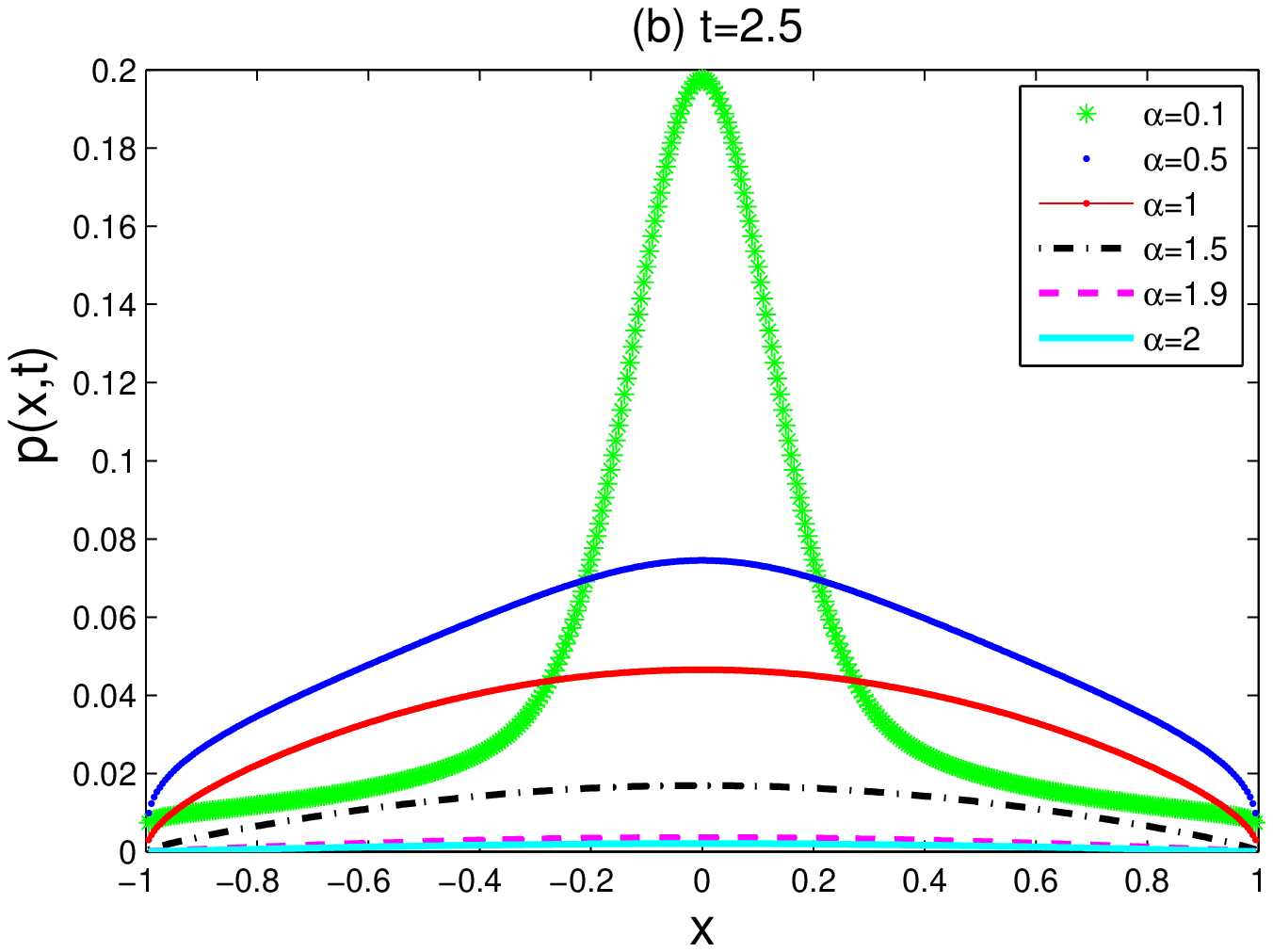}
\end{center}
\caption{The dependence of the probability density function $p$ on
the values of $\a$. The initial condition is the Gaussian profile
$\displaystyle p(x,0)= \sqrt{\frac{40}{\pi}} e^{-\frac{x^2}{40}}$
peaked at the origin. The probability densities $p$ at time (a)
$t=0.25$ and  (b) $t=2.5$,  for $d=0,f\equiv 0$, $\eps =1$,
$h=0.001$, $\Delta t=0.5 h^\a$ with the absorbing condition.
 }
\label{fig.dag}
\end{figure}
Next, we examine the dependence of the probability density $p$ on
the values of $\alpha$ (different L\'evy measures) for symmetric
$\a$-stable L\'evy motion ($d=0$ and $f\equiv 0$).
Figure~\ref{fig.dag} shows the density $p$ as a function of the
position $x$ for $\a=0.1, 0.5, 1, 1.5, 1.9, 2$ at two different
times $t=0.25$ and $t=2.5$, starting from the same Gaussian
distribution $\displaystyle p(x,0)= \sqrt{\frac{40}{\pi}}
e^{-\frac{x^2}{40}}$.

If the particle starts near the center of the domain (the origin
$x=0$ in our configuration), the numerical results in
Fig.~\ref{fig.dag} show that, for $x$ near the origin, the
probability $p$ is smaller as the value of $\a$ increases,
suggesting that the particle is less likely to stay near the origin
when the value $\a$ is larger. At earlier times, for example
$t=0.25$, the probability of finding the particle at positions $x$
away from the origin is larger when $\a$ is larger (except for the
special Gaussian noise case $\a=2$). Note that the probability
density for the Gaussian diffusion ($\a=2$) is much smaller than
those for the symmetric $\a$-stable L\'evy motion with $\a<2$. The
profiles of the density function change to parabola-like shape,
having negative second-order derivative, for large values of $\a$
($\a \ge 1.5$ at time $t=0.25$ in our case),
 while the profiles stay bell-type shape for smaller values of $\a$.
At later times, for example $t=2.5$, the probability near the
boundary of the domain becomes a decreasing function of $\a$ for the
values of $0.5\leq \a \leq 2$, and the profiles of the density
functions become all parabola-like except for the smallest value of
$\a=0.1$. It is interesting to point out that the probability $p$
appears to be discontinuous at the boundary $x=\pm 1$ for $\a=0.1$
and $0.5$. Examining the solutions more closely near the boundaries,
we find that the probability densities are actually continuous but
have very sharp transition (a thin "boundary layer") at the
boundaries. However, in our earlier work \cite{Ting12}, the
numerical results show that the mean first exit time of symmetric
$\a$-stable L\'evy motion is discontinuous at the boundary of a
bounded domain for $\a<1$.

\begin{figure}
\begin{center}
\includegraphics*[width=7.9cm]{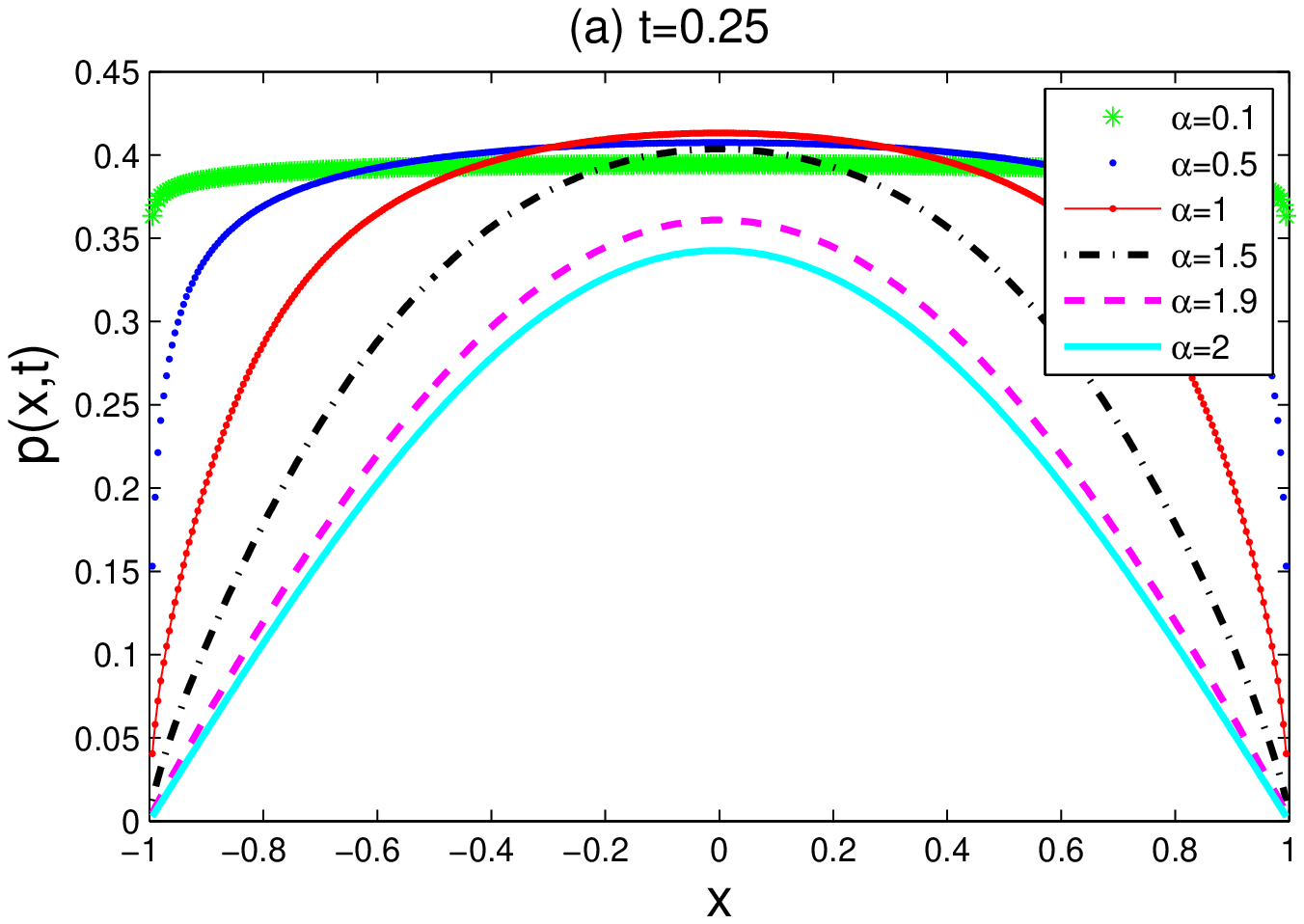}
\includegraphics*[width=7.9cm]{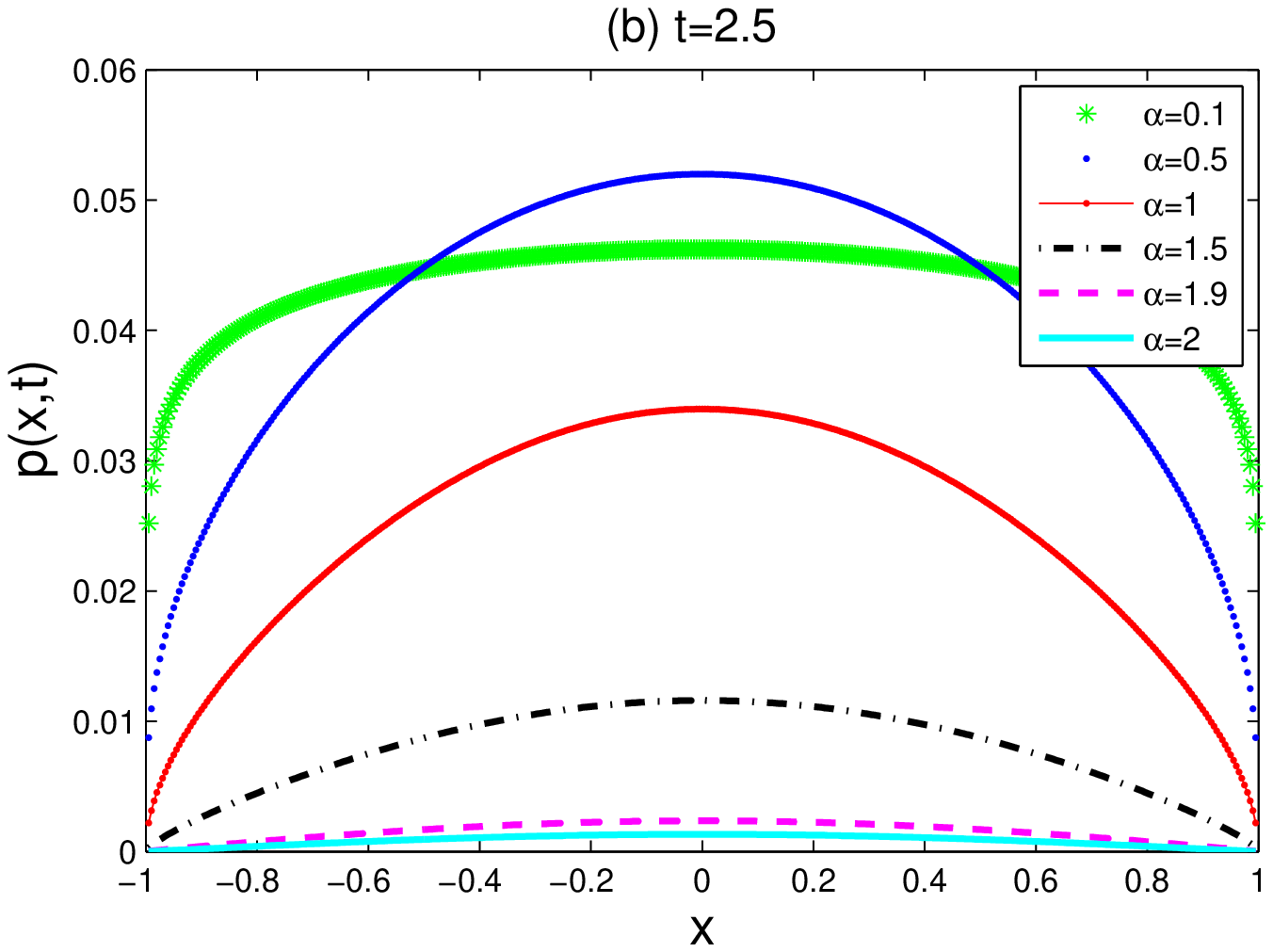}
\end{center}
\caption{The dependence of the probability density function $p$ on
the values of $\a$. The initial condition is the uniform profile
$p(x,0)= 0.5 I_{-1<x<1}$. The probability densities $p$ at time  (a)
$t=0.25$ and (b) $t=2.5$,   for $d=0,f\equiv 0$, $\eps =1$,
$h=0.001$, $\Delta t=0.5 h^\a$ and the absorbing condition.   } 
\label{fig.dau}
\end{figure}
On the other hand, when the initial condition is uniformly
distributed (i.e., the particle is equally likely starting from any
point in the interval $[-1,1]$), the results in Fig.~\ref{fig.dau}
show that the probability of finding the "particle" near the center
of the domain is a non-monotone function of $\a$, while the
probability near the boundary of the domain is a decreasing function
of $\a$ for all values of $\a$ in $(0,2]$. The profiles of the
probability function changes from the starting shape of a step
function to that of parabola-like for $\a \geq 1$. In contrast, for
small values of $\a$ (such as $\a=0.1$), the probability function
keeps its flat profile as time elapses. Again, the functions are
continuous at the boundary $x=\pm 1$ after a careful examination of
the solutions near the boundary.

%\subsubsection{Effects of L\'evy noises}

\begin{figure}[h]
\begin{center}
\includegraphics*[width=12cm]{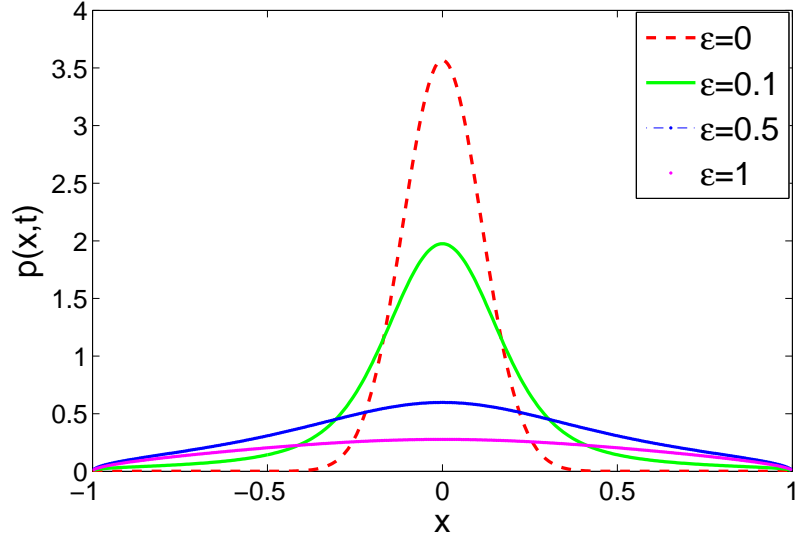}
\end{center}
\caption{Dependence of the probability density $p$ at $t=1$ on the magnitude
of L\'evy noises 
for the case of $f\equiv 0$, pure non-Gaussian noise ($d=0$), 
$\alpha=1$, $D=(-1,1)$ with the absorbing condition and
 different sizes of $\varepsilon=0, 0.1, 0.5, 1$. 
}{\label{eps}}
\end{figure}
Next, we investigate the effect of the magnitude of L\'evy noises. 
Figure~\ref{eps} shows the probability density function $p$
at time $t=1$ starting from the initial profile $\displaystyle p(x,0)=
\sqrt{\frac{40}{\pi}} e^{-\frac{x^2}{40}}$ for the case of
pure non-Gaussian noise $d=0$, $\alpha=1$, the domain $D=(-1,1)$
and no deterministic driving force $f\equiv 0$ but different magnitude
of $\varepsilon=0,0.1,0.5,1$. When the noise is absent $\varepsilon=0$,
the probability density keep the initial profile $p(x,t)=p(x,0)$. 
As the amount of noise $\varepsilon$ increases, the density function
becomes smaller at the center of the domain and the function profile becomes
flatter. Comparing with Figs.~\ref{time_sequence}(a) and \ref{fig.dag}(a), 
we find that increasing the magnitude of L\'evy noises 
has the similar effect as
increasing the time or the value of $\alpha$.

\subsubsection{Effect of the deterministic driving force due to
Ornstein-Uhlenbeck(O-U) potential: $f(x)=-x$}

\begin{figure}
\begin{center}
\includegraphics*[width=12cm]{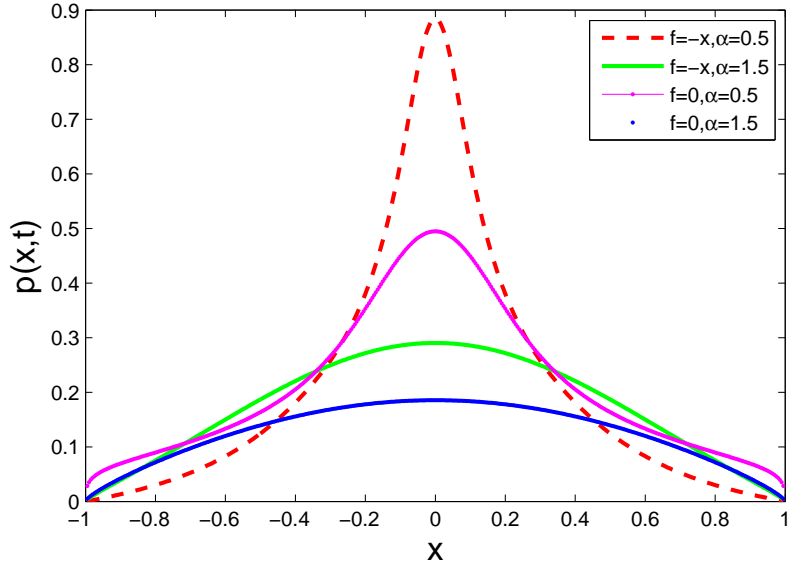}
\end{center}
\caption{The effect of the deterministic driving force $f$. The
plots show the probability density functions $p$ at time $t=1$ with
$d=0,\eps=1, \a=0.5, 1.5$ and the absorbing condition 
starting from a Gaussian profile when the
O-U potential $f(x)=-x$ is present and when $f=0$. } \label{f_alpha}
\end{figure}
Next, we examine the effect of a deterministic driving force, the
O-U potential $f(x)=-x$, on the probability density function $p$.
The initial condition is the Gaussian profile $\displaystyle p(x,0)=
\sqrt{\frac{40}{\pi}} e^{-\frac{x^2}{40}}$ as in Fig.~\ref{fig.dag}.
We keep other parameters the same as those in Fig.~\ref{fig.dag}:
$d=0$ and $\eps=1$. The O-U potential drives the "particle" toward
the origin ($x=0$), the center of the domain, while the stochastic
term in the SDE~\eqref{sde999} acts like a "diffusional" force.
Figure~\ref{f_alpha} gives us the result of the competition between
the two forces. At the smaller value of $\a=0.5$, we compare the
density functions $p$ at time $t=1$ with the driving force $f(x)=-x$
and without any deterministic force $f\equiv 0$. As expected, we
find that the probability of finding the particle near the origin or
the boundary is higher or lower respectively when the O-U potential
is present. At the larger value of $\a=1.5$, the probability near
the center is higher when the O-U potential is present, while the
probabilities are almost identical near the boundary whether
$f(x)=-x$ or $f=0$.

\begin{figure}
\begin{center}
\includegraphics*[width=7.9cm]{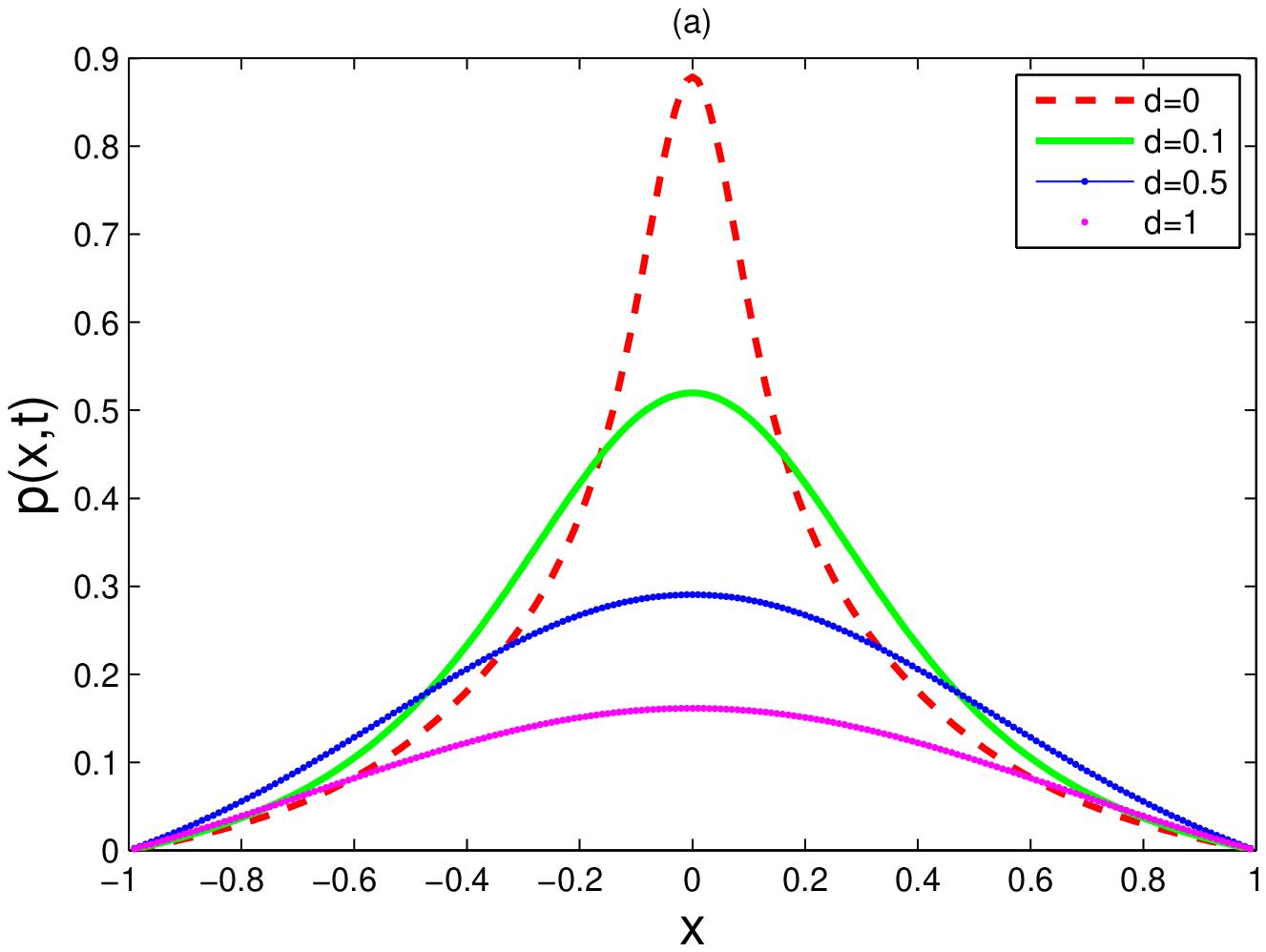}
\includegraphics*[width=7.9cm]{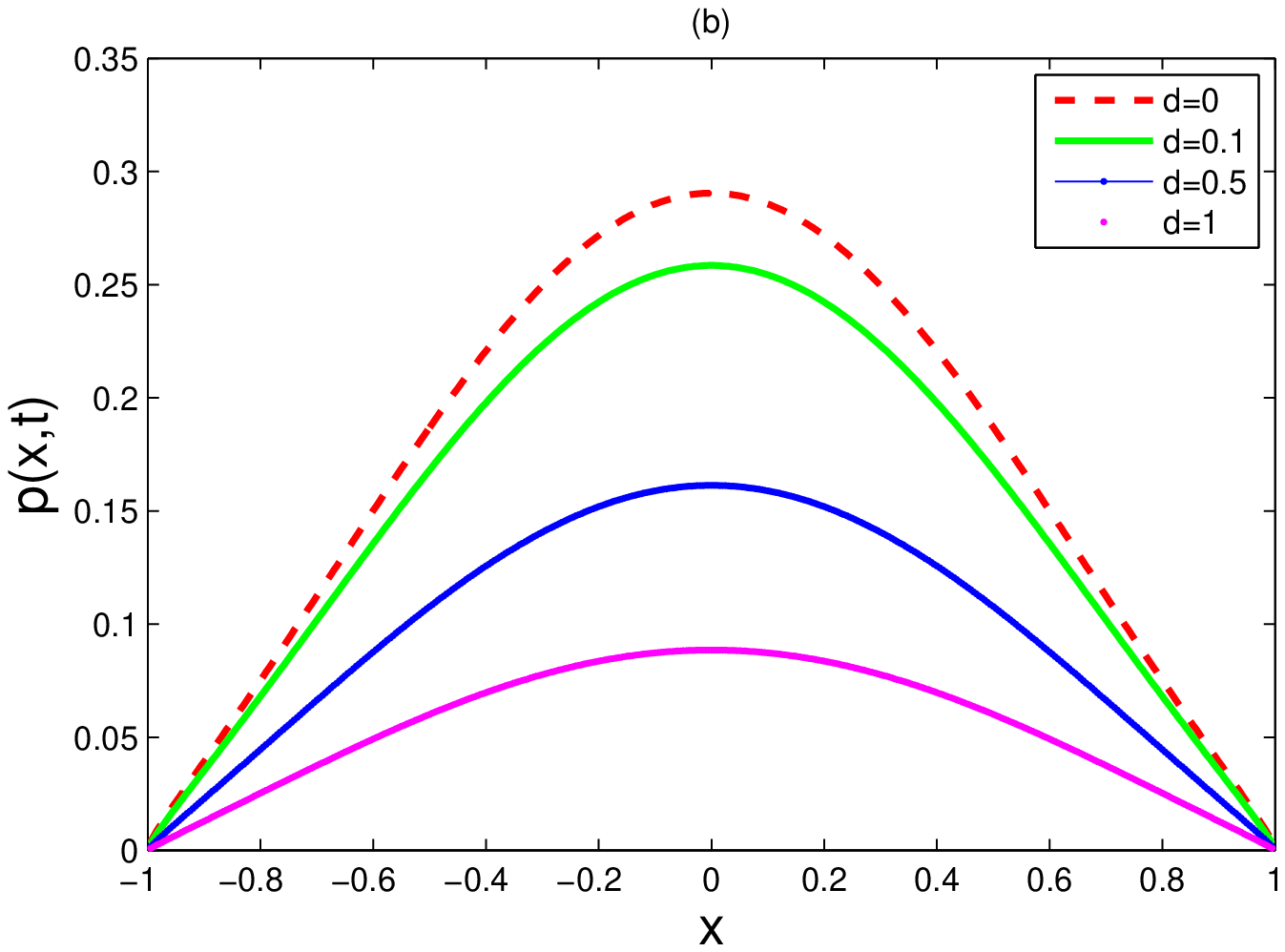}
\end{center}
\caption{The effect of adding Gaussian noise($d\ne 0$). The
probability density function $p$ at time $t=1$ is shown for
$f(x)=-x$, $\eps=1$, the absorbing condition
 and different values of $d$, 
starting from $\displaystyle p(x,0)=
\sqrt{\frac{40}{\pi}} e^{-\frac{x^2}{40}}$: (a)  $\a=0.5$ and (b)
$\a=1.5$. }{\label{O_U}}
\end{figure}
As expected, adding the Gaussian noise $d \neq 0$ would decrease the
probability density function. Figure~\ref{O_U} shows the density
functions $p$ at time $t=1$ with the O-U potential $f(x)=-x$ and
$\eps=1,\a=0.5, 1.5$ for different amount of Gaussian noises
$d=0,0.1,0.5,1$. For the smaller value of $\a=0.5$, shown in
Fig.~\ref{O_U}(a), the density function profiles change from a
bell-like shape to a parabola-like one as the Gaussian noise $d$
increases. However, for the larger value $\a=1.5$, the density
profiles at $t=1$ are parabola-like regardless the value of $d$. As
we have observed earlier, the L\'evy noise with larger values of
$\a$ has similar diffusive effect as the Gaussian noise on the
probability density function.

%\begin{figure}
%\begin{center}
%\includegraphics*[width=12.5cm,height=7.5cm]{./figures/small_diff.eps}
%\includegraphics*[width=12.5cm,height=7.5cm]{./figures/large_diff.eps}
%\end{center}
%\caption{$f=-x,\varepsilon =1$ with space step size $h=1/50$. Left:
%$d=0.1$ and right: $d=1$. }
%\end{figure}

%\begin{figure}[]
%\begin{center}
%\includegraphics*[width=6.7cm,height=5.5cm]{./figures/diff_and_f_1.eps}
%\includegraphics*[width=6.7cm,height=5.5cm]{./figures/diff_and_f_2.eps}
%\includegraphics*[width=6.7cm,height=5.5cm]{./figures/diff_f_1.eps}
%\includegraphics*[width=6.7cm,height=5.5cm]{./figures/diff_f_2.eps}
%\end{center}
%\caption{Upper two: $d=1,f=-x,\varepsilon =1$; Lower two:
%$d=0,f=-x,\varepsilon =1$}
%\end{figure}

\subsubsection{Dependence on the size of the domain}

\begin{figure}[h]
\begin{center}
\includegraphics*[width=9.0cm]{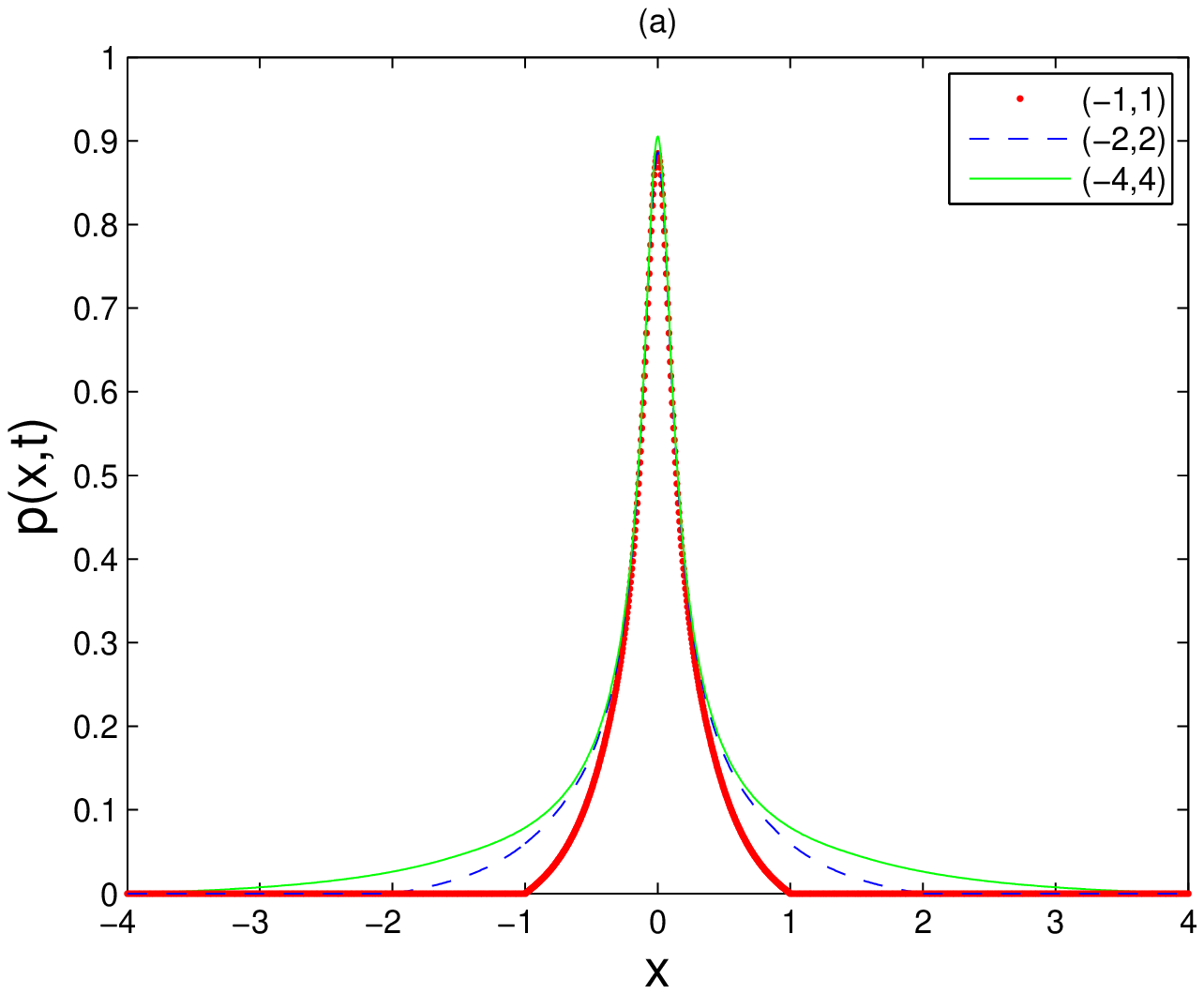}
\includegraphics*[width=9.0cm]{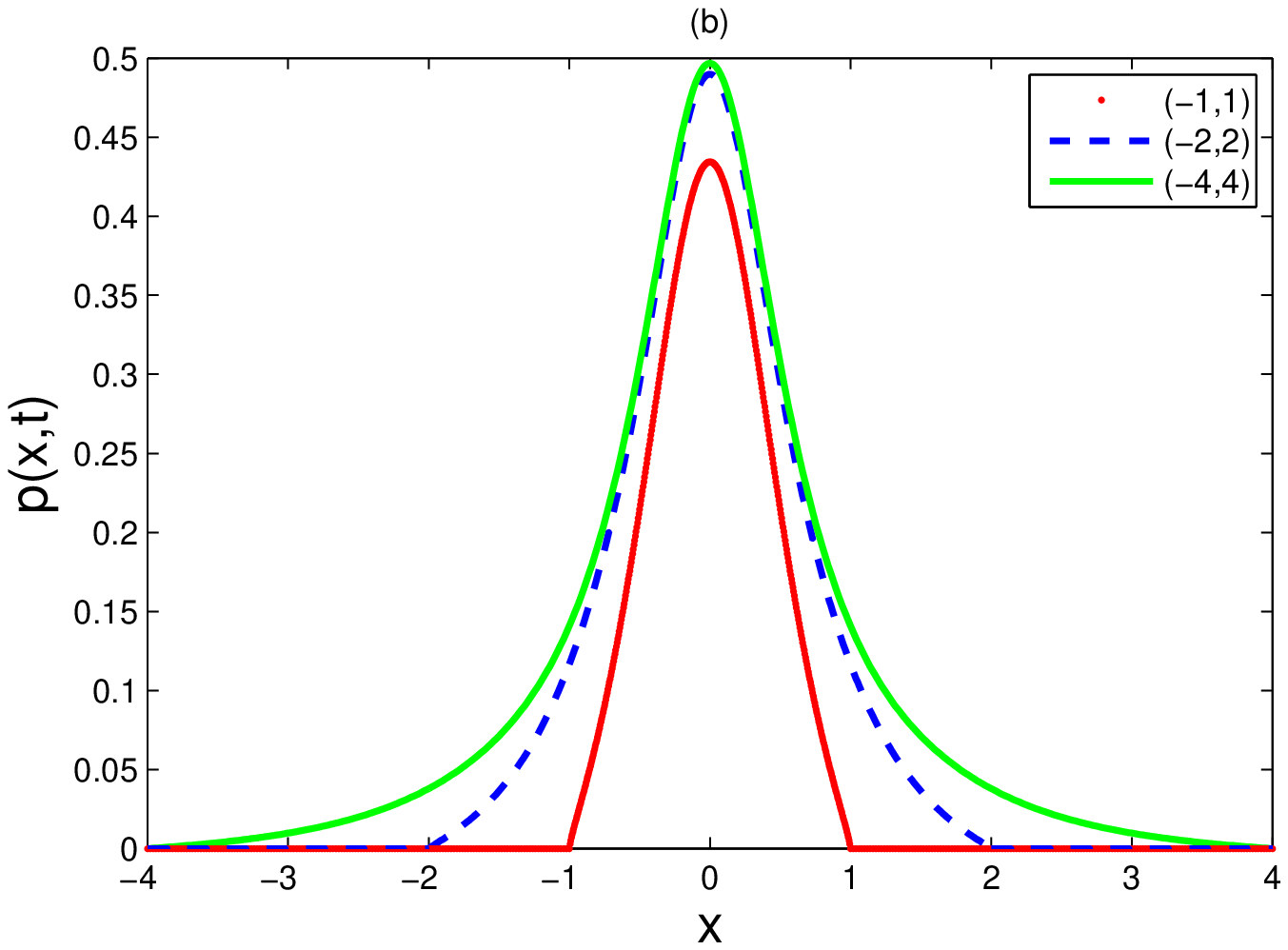}
\includegraphics*[width=9.0cm]{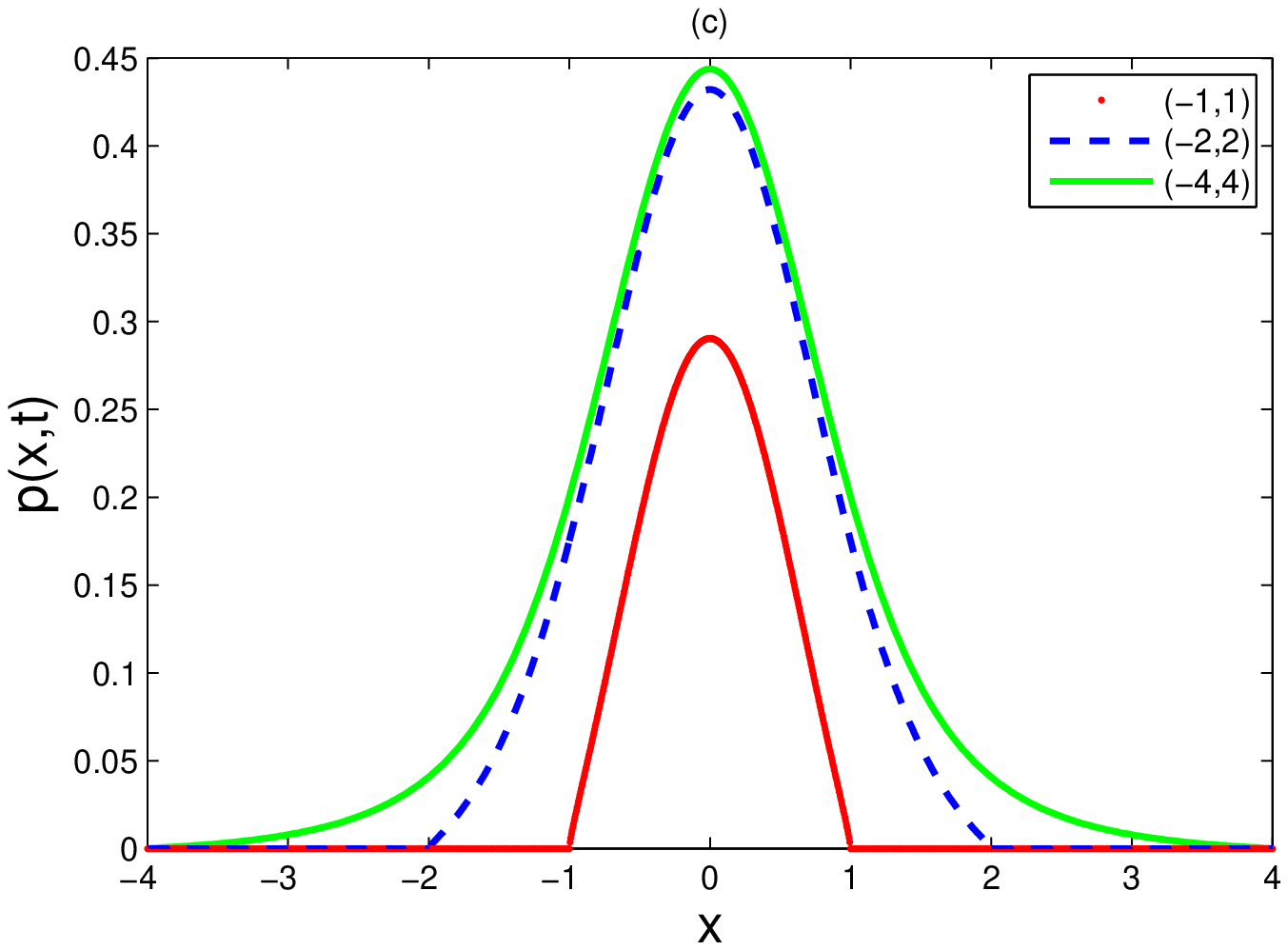}
\end{center}
\caption{Dependence of the probability density at $t=1$ on the size
of domain $D$ for the case of Ornstein-Uhlenbeck potential $f(x)=-x$
with pure non-Gaussian noise ($d=0$), $\varepsilon=1$, the absorbing
condition and different domain sizes $(-1,1)$, $(-2,2)$ and $(-4,4)$:
 (a)  $\a=0.5$; (b)  $\a=1$;  and (c) $\a=1.5$. }
\label{domain_size}
\end{figure}
In Fig.~\ref{domain_size}, we plot the probability densities at $t=1$ for three 
different sizes of domains $(-1,1)$, $(-2,2)$ and $(-4,4)$ and three different
values of $\a=0.5,1$ and $1.5$. Here, the O-U potential $f(x)=-x$
is present but there is no Gaussian noise $d=0$, and the initial 
profile of the probability density is the same Gaussian one as before.
At a fixed value $x$, 
the density function has higher value as the domain size increases due 
to the fact that it is harder to escape a larger domain. 
For the smaller value of $\a=0.5$, Fig.~\ref{domain_size}(a) shows
that the density function values are relatively insensitive 
to the domain sizes in the regions near the center, 
the region $(-0.5,0.5)$ in this case.
It suggests that, for small values of $\a$, 
the "particle" escape the bounded domains mostly via large jumps rather 
than small movement. For the two other values of $\a=1$ and $1.5$,
the density function is smaller for the smaller domain $D=(-1,1)$,
but the probability functions have similar values near the center 
for the two larger domains $(-2,2)$ and $(-4,4)$. The results shown 
in Fig.~\ref{domain_size} imply that, 
for each value of $\a$, there is a threshold for the size of
the domain beyond which the density functions will have similar values 
in the regions centered at the origin.

\subsection{Natural condition}

The distribution and the properties of
$\alpha-$stable random variables are well known. However, the time evolution
of the probability density of symmetric $\alpha-$stable L\'evy process is
less well studied.
We approximate the solution to the Fokker-Planck equation \eqref{FPE555} 
subject to
the natural far-field condition \eqref{eq.nbc} by using the scheme \eqref{eq.sdn} with increasing domain sizes until
the PDF $p$ converges on the intervals of our interests. 

To evaluate the error caused by replacing the infinity domain $(-\infty,
\infty)$ for the natural condition with finite domains,
we compute the integral 
$\displaystyle I_p(t)\equiv\int_{-\infty}^\infty p(x,t) {\rm d}x$.
For the natural far-field condition, the integral $I_p(t)$ keeps 
its initial value of one for all times. 
For the case of $f\equiv 0, d=0$, $\varepsilon=1$,
$\a=1$, we compute the integral $I_p(1)$ numerically using trapezoidal rule 
based on the numerical solutions with 
increasing domain sizes $(a,b)=(-5,5), (-10,10),
(-20,20)$, $(-40,40)$, $(-80,80)$.
The corresponding values of the integral 
for the different domain sizes are $0.997060069$, $0.999318$,
$0.999838$, $0.999961$, $0.999990$ respectively. It is obvious that
the integrals for these large domains are close to one,
implying that the truncating the original infinite domain 
cause small amount of error. 
The integral is closer to one as the domain size becomes
larger. The convergence order of the integral with respect to
the domain size is about $2$.

%\begin{figure}[]
%\begin{center}
%\includegraphics*[width=12cm,height=8cm]{./figure/tgao7_natural_2.eps}
%\end{center}
%\caption{$d=0,f=-x,\varepsilon =1$,$\alpha=1.5$}
%\end{figure}

\begin{figure}[h]
\begin{center}
\includegraphics*[width=8.0cm]{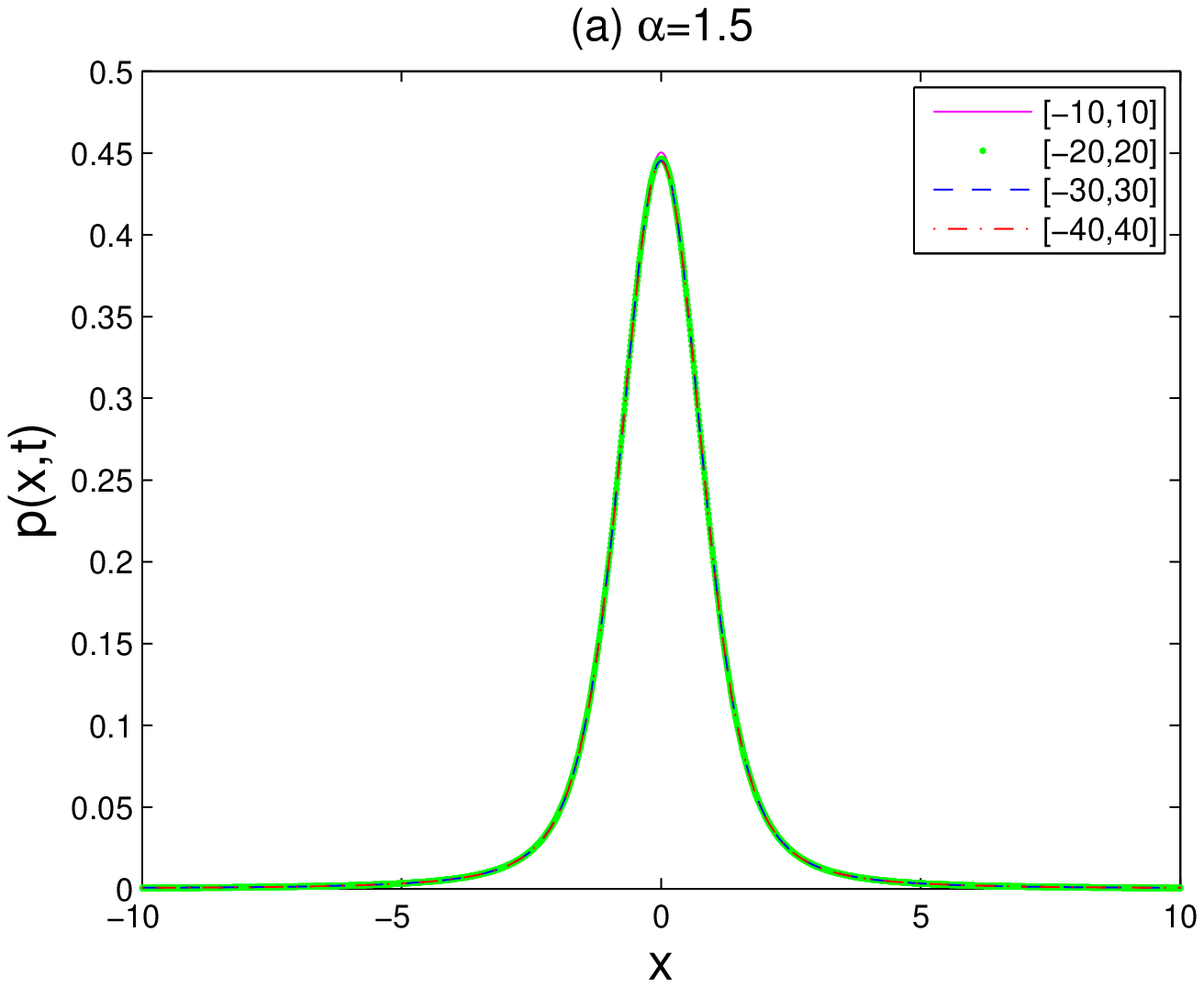}
\includegraphics*[width=8.0cm]{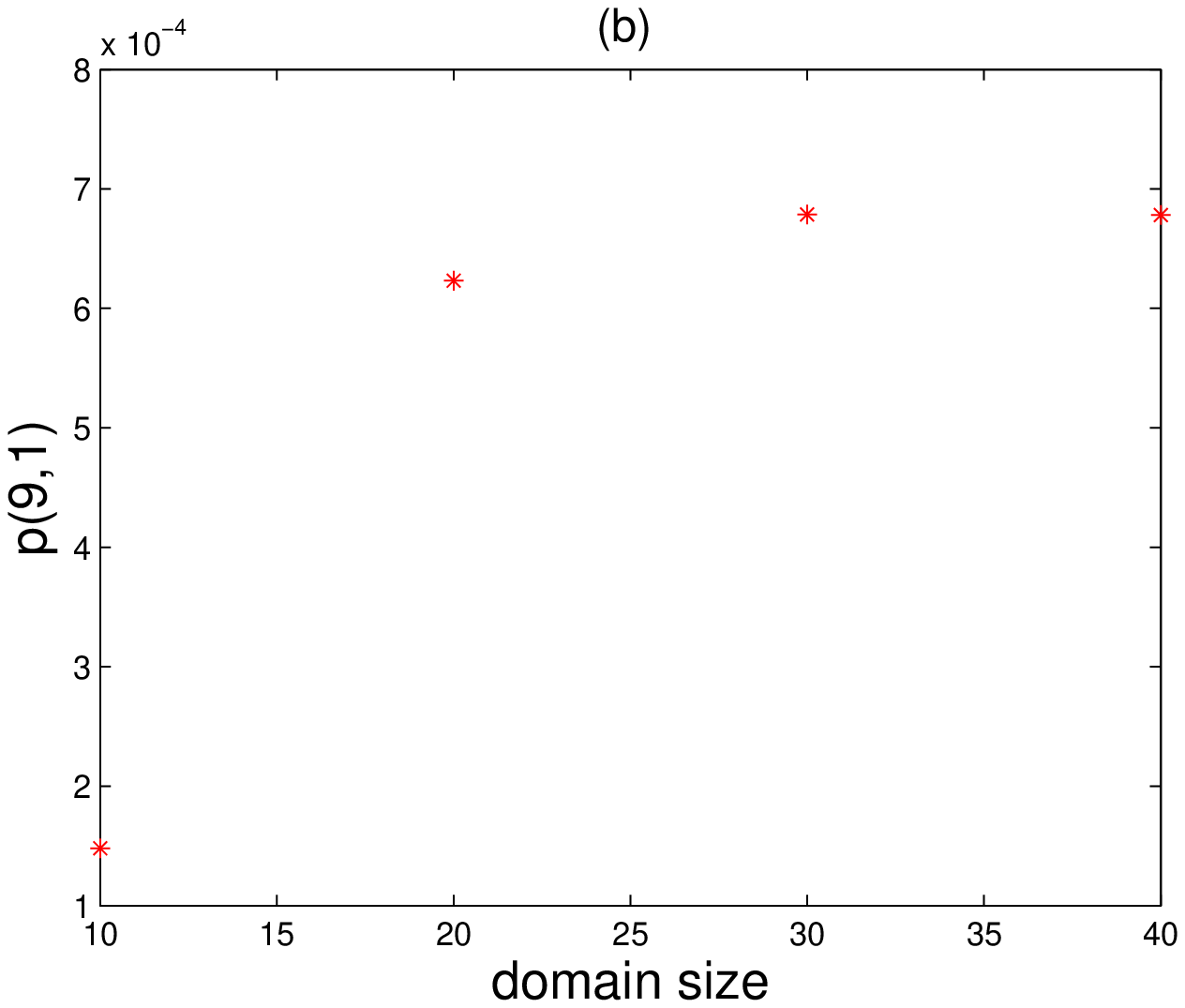}
\includegraphics*[width=8.0cm]{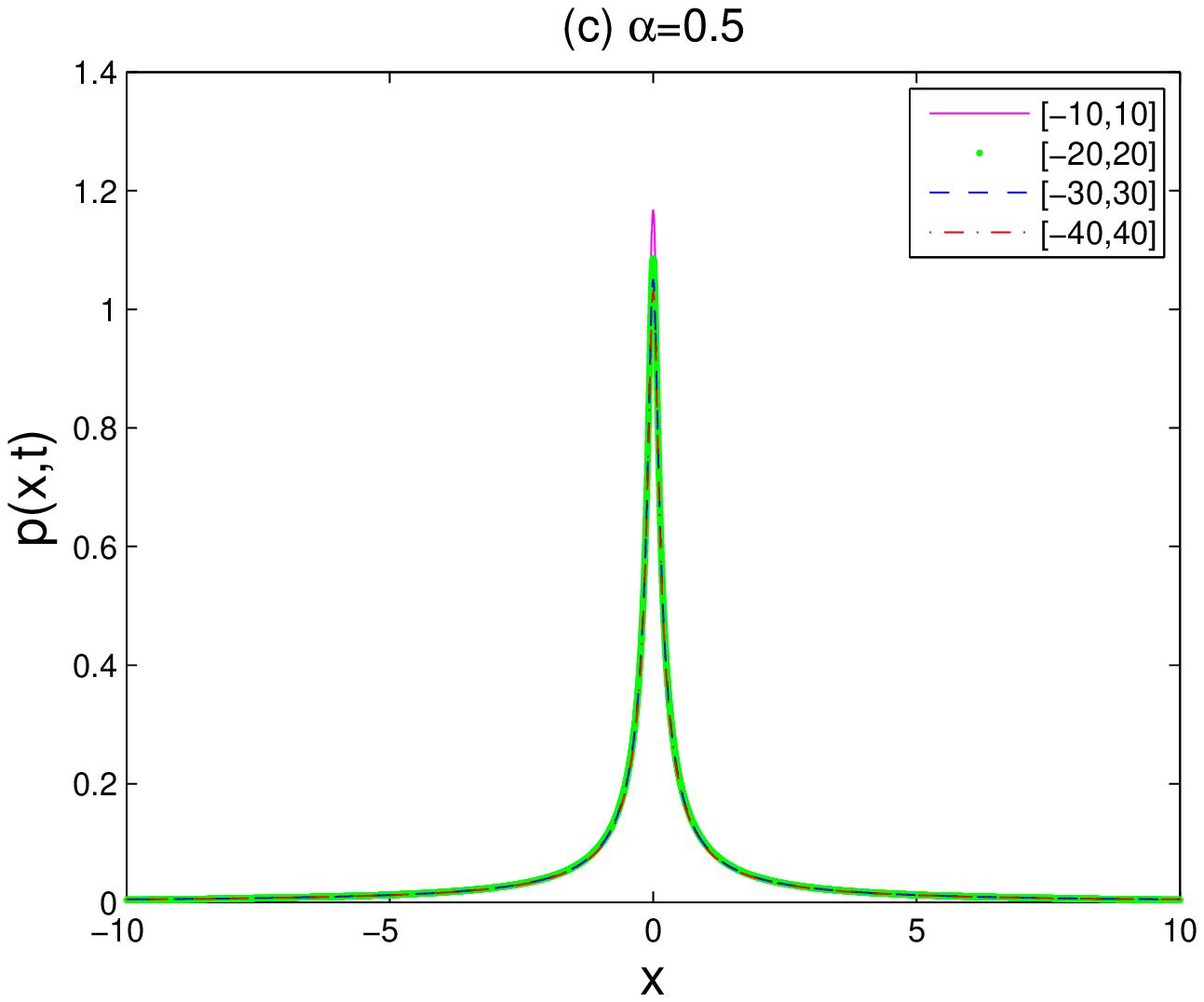}
\includegraphics*[width=8.0cm]{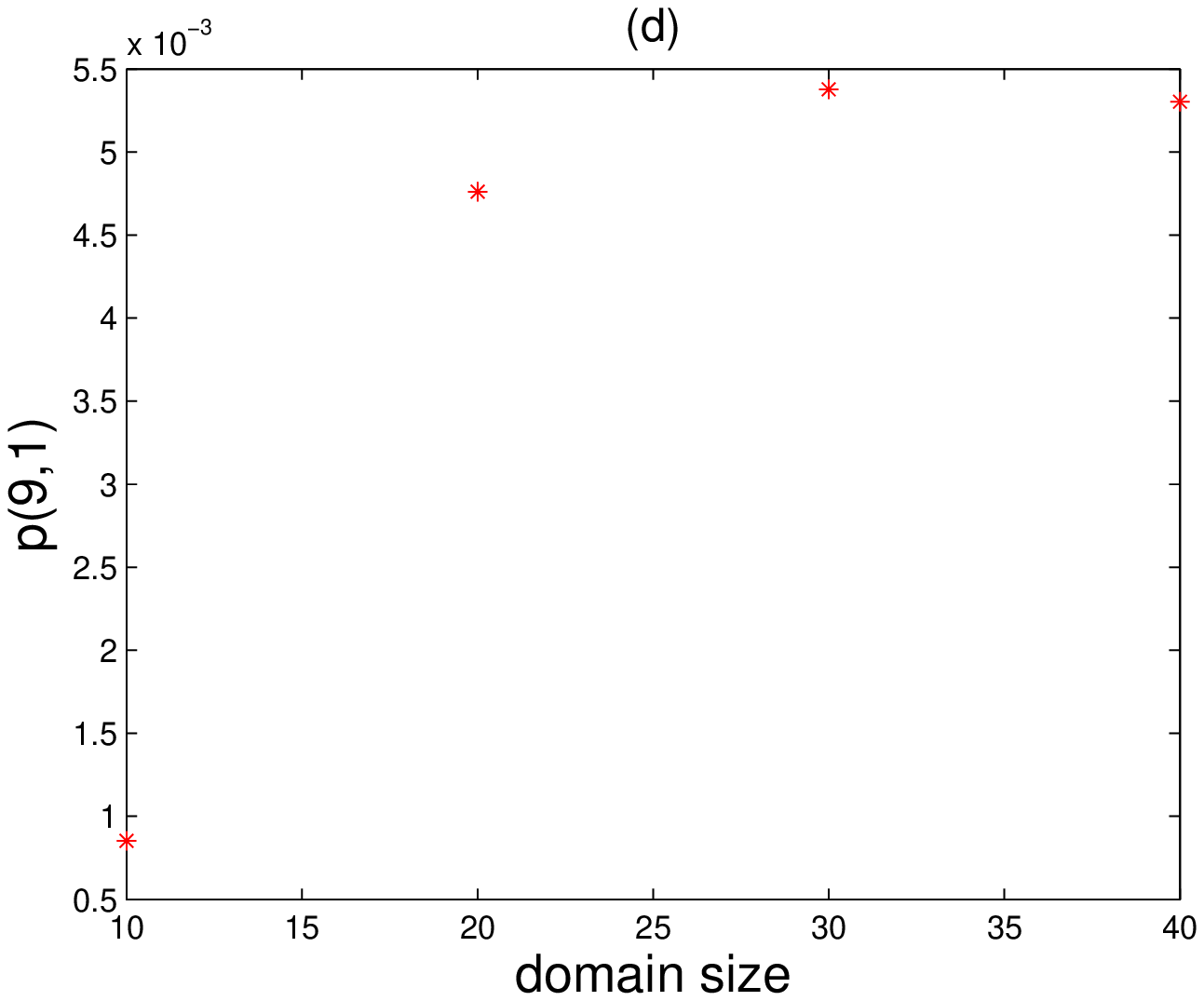}
\end{center}
\caption{The probability density function $p$ at time $t=1$ 
with the natural condition, 
$f(x)=-x$, $\eps=1$, $\alpha=1$, and the initial Gaussian profile.
(a) The numerical solutions for $\alpha=1.5$
from different computational domains
$[-10,10], [-20,20],[-30,30],[-40,40]$.
(b) The density function values at the fixed point $(x,t)=(9,1)$
plotted against the computational domain size $(-L,L)$ for $\alpha=1.5$.
(c) Same as (a) except for $\alpha=0.5$. 
(d) Same as (b) except for $\alpha=0.5$. } 
\label{fig.nb}
\end{figure}
Starting with the same initial Gaussian profile 
$p(x,0)= \sqrt{\frac{40}{\pi}} e^{-\frac{x^2}{40}}$, $d=0$,
$f(x)=-x$, $\eps=1$ and $\a=1.5$, Fig.~\ref{fig.nb}(a) shows the PDF
$p$ on the interval $(-10,10)$ at time $t=1$ using increasing 
computational domain sizes 
$(-L,L)=(-10,10), (-20,20)$, $(-30,30)$, $(-40,40)$. Note that
the integro-differential equation \eqref{FPE555} for the natural
condition is different than Eq.~\eqref{fpe1Dn3} 
for the absorbing condition. 
Consequently, the corresponding numerical schemes 
are also different: \eqref{eq.sdn} for the natural far-field condition
and \eqref{nm1D3} for the absorbing condition. 
The numerical results show that the four
different approximations fall onto each other, indicating that the
probability $p$ becomes independent of domain size when the domain
is large enough. In Fig.~\ref{fig.nb}(b), we demonstrate the
convergence of the value of the probability $p$ at the point
$(x,t)=(9,1)$. In this case, the values of $p$ on the interval
$(-10,10)$ do not change significantly when the computational
domain $(-L,L)$ includes $(-20,20)$. Consequently, on the finite interval
$(-10,10)$, we can argue that the profile of $p$ obtained by using the
computational domain larger than $(-20,20)$ shall agree well 
with that for natural far-field condition $(-\infty,\infty)$. 
Comparing Fig.~\ref{fig.nb}(a) with Fig.~\ref{domain_size}(c),
 we point out that the values of the density function
for the natural condition near the center are slightly larger 
than those for the absorbing condition, when the other conditions match with
each other.  

Figure~\ref{fig.nb}(c) and (d) show the PDF for $\alpha=0.5$, keeping
the other conditions the same. Again, the numerical solution converges
for large enough computational domain. However, with other conditions
being identical, the values of the PDF for natural condition
near the center, shown in Fig.\ref{fig.nb}(c), are significantly 
higher than those for the absorbing condition 
shown in Fig.~\ref{domain_size}(a).   

\begin{figure}[h]
\begin{center}
\includegraphics*[width=16.8cm]{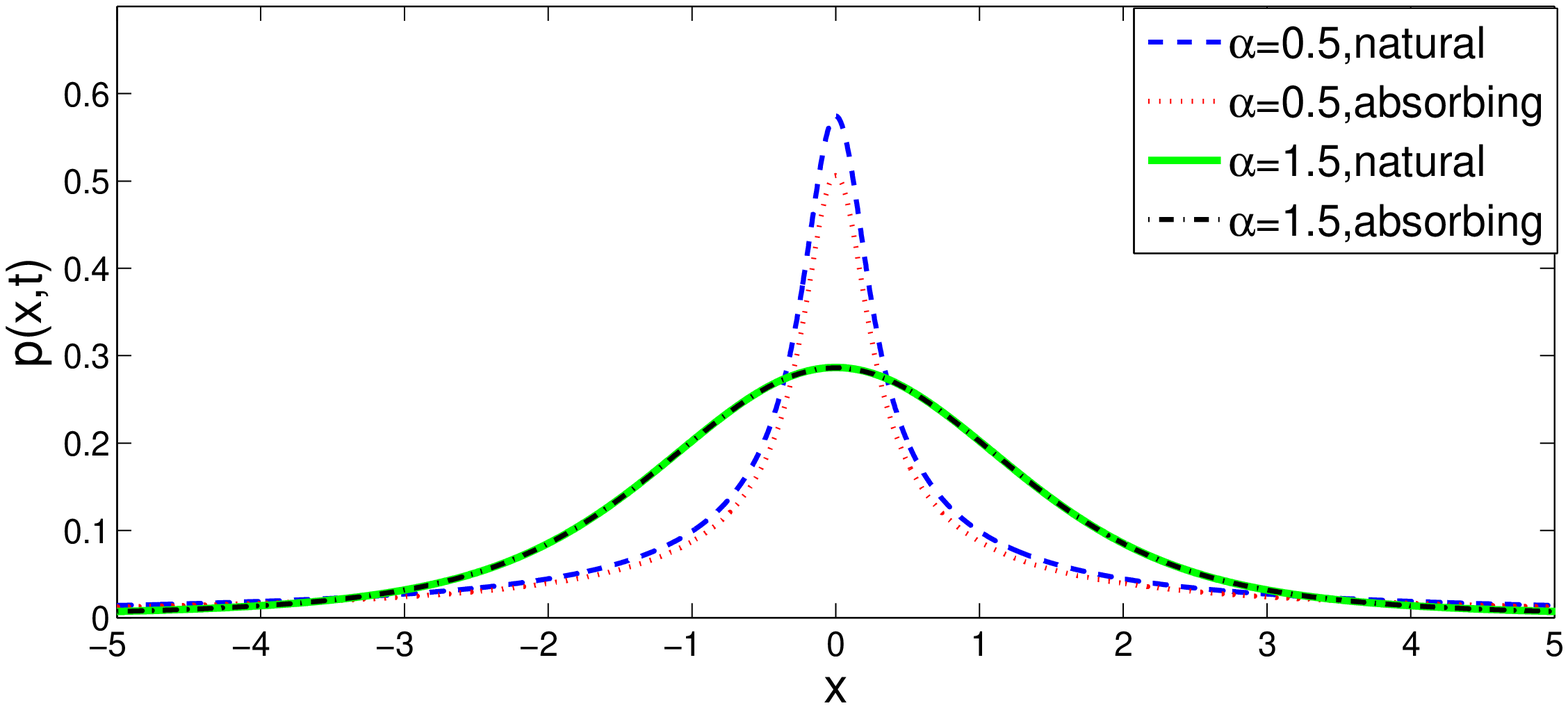}
\end{center}
\caption{Comparison of different auxiliary conditions. 
The probability density functions at time $t=1$ starting with 
a Gaussian initial condition for $d=0$, $f\equiv 0$, $\eps=1$
and $\alpha=0.5$ and $\alpha=1.5$. }{\label{CompBC}}
% when alpha=0.5 dt=0.5h
%when alpha=1.5 dt=0.4h^1.5
\end{figure}
The conclusions from comparing the two auxiliary conditions are the same
when we remove the deterministic driving force $f$. Figure~\ref{CompBC}
shows the PDFs for both auxiliary conditions on the same graph.

%\textcolor{blue}{Ting: could you compute the solution with domain
%size $(a,b)=(-40,40)$ for $d=0, f=-x, \eps=1, \a=0.1, 0.5, 1, 1.5,
%1.9, 2$ at $t=1$? Plot these solutions in one figure. Make sure the
%results converge on the interval $(-10,10)$ for each $\a$ (domain
%large enough)}

%Figure~\ref{natural_alpha}.

%\begin{figure}[]
%\begin{center}
%\includegraphics*[width=12cm]{./figure/natural_OU.eps}
%\includegraphics*[width=12cm,height=8cm]{./figure/tgai7_natural.eps}
%\includegraphics*[width=6.7cm,height=5cm]{./figure/tgao7_natural_whole.eps}
%\end{center}
%\caption{$d=0,f=-x,\varepsilon =1$,$h=1/100$ with different
%$\alpha$. Note: computed on [-10,10] with natural BC and
%T=1}{\label{natural_alpha}}
%\end{figure}

%\begin{figure}[]
%\begin{center}
%\includegraphics*[width=6.7cm,height=5.5cm]{./figures/withnocorrection_d.eps}
%\includegraphics*[width=6.7cm,height=5.5cm]{./figures/withnocorrection_d_comp.eps}
%\end{center}
%\caption{$d=1$, $f=-x$, $\varepsilon =1,\alpha=0.5$. Left: pdf
%functions on  $[-10,10], [-20,20],[-30,30],[-40,40]$ respectively;
%Right: value of pdf functions at $(x,t)=(9,1)$.}
%\end{figure}

%\begin{figure}[]
%\begin{center}
%\includegraphics*[width=6.5cm,height=5cm]{./figures/withnocorrection.eps}
%\includegraphics*[width=6.5cm,height=5cm]{./figures/withcorrection.eps}
%\end{center}
%\caption{$d=0$, $f=-x$, $\varepsilon =1, \alpha=0.5$. Left: with no
%correction on the boundary; Right: With correction on the boundary.}
%\end{figure}

\subsubsection{Deterministic force driven by the double-well potential: 
$f(x)=x-x^3$}
\begin{figure}[h]
\begin{center}
\includegraphics*[width=14.7cm]{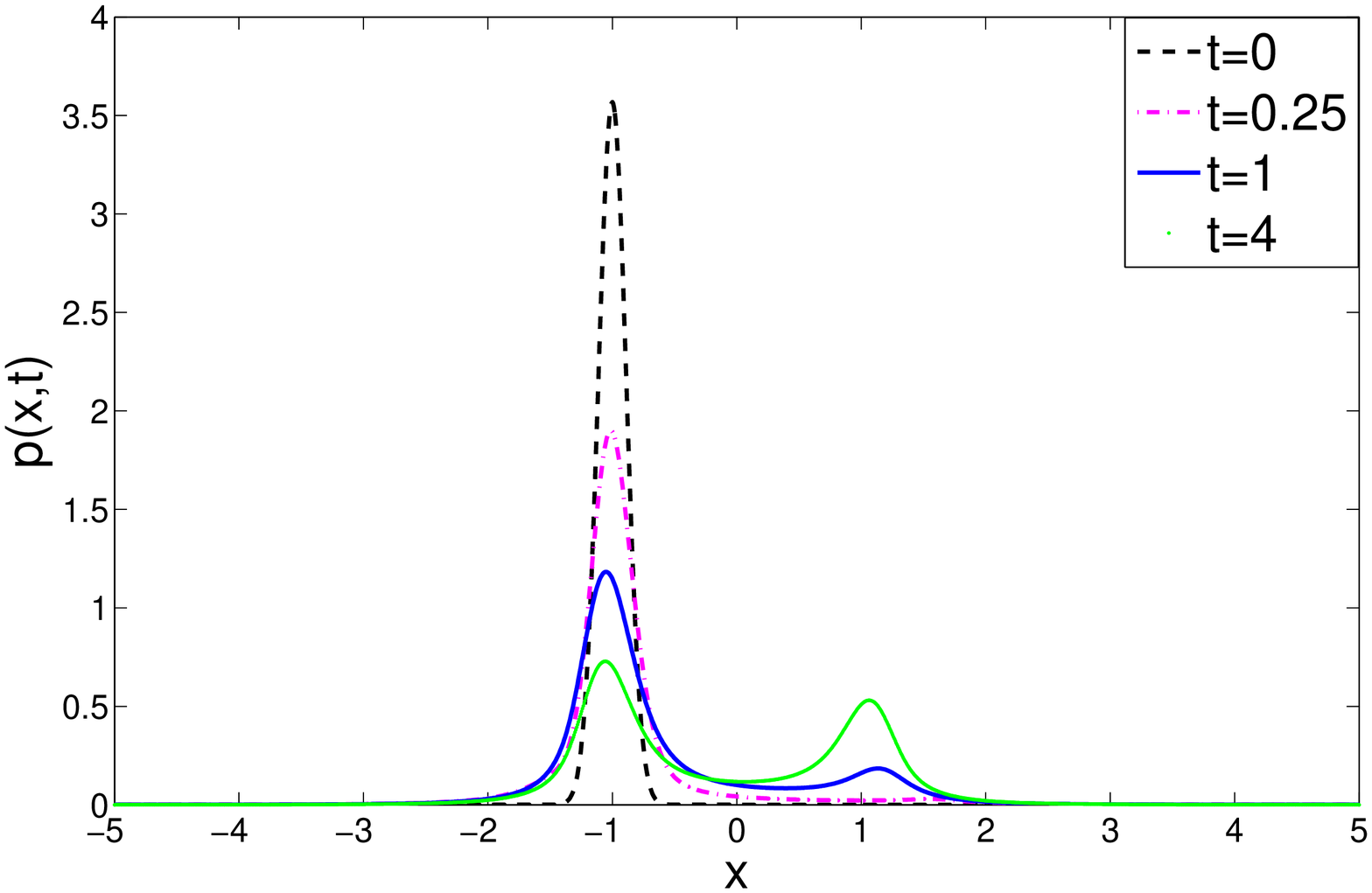}
\includegraphics*[width=14.7cm]{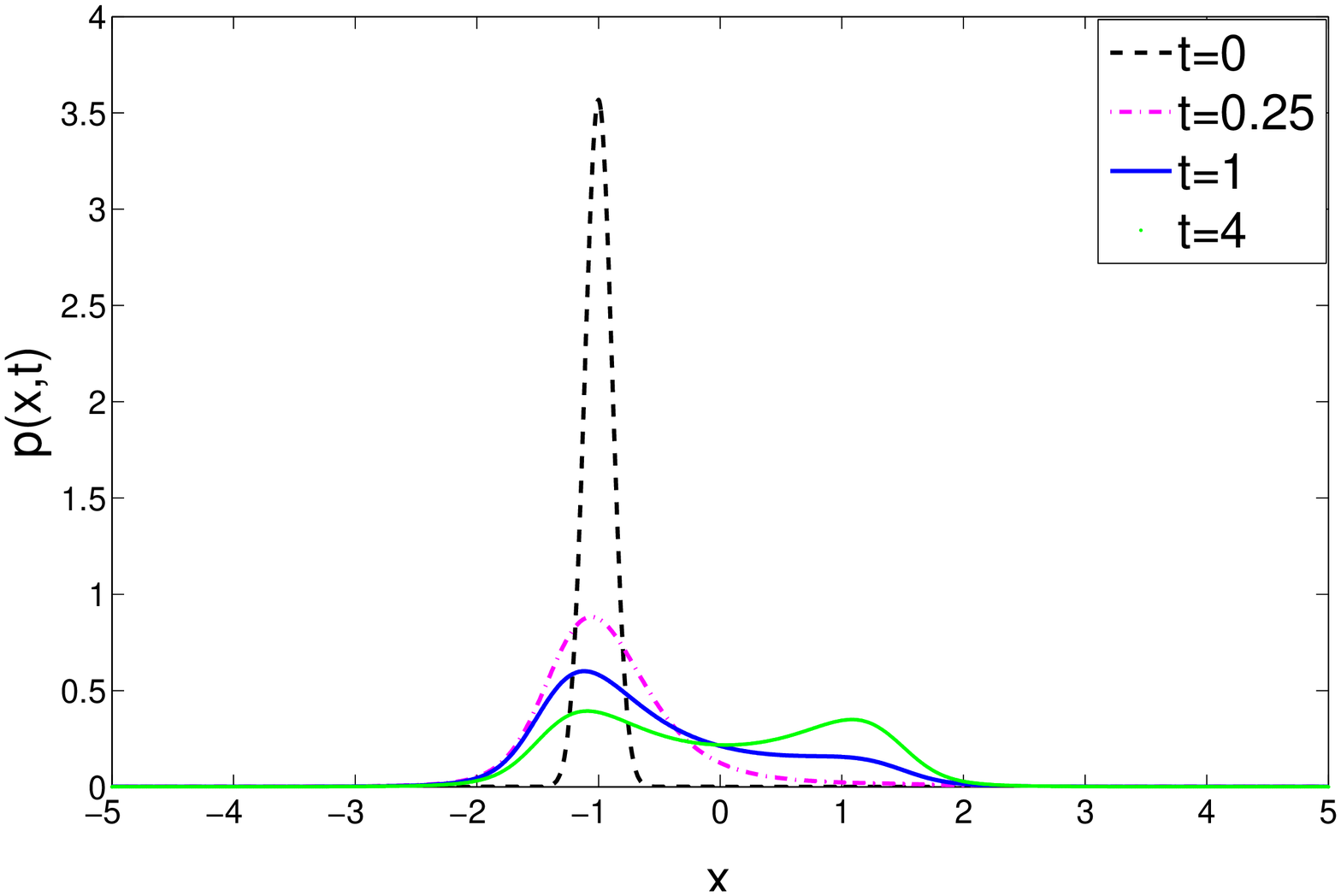}
\end{center}
\caption{The probability density functions for deterministic driving force 
$f(x)=x-x^3$, $d=0.1$, $\eps=1$ and $\alpha=0.5$(top), $\a=1.5$(bottom).}  
\label{Dwell}
\end{figure}
The deterministic force governed by the double-well potential, $f(x)=x-x^3$, 
drives the 'particle' toward the two stable states $x=\pm 1$. 
The potential has a local maximum located at the center $x=0$
and two local minima symmetrically located at $x=\pm 1$.
In the absence of noises, the 'particle' would be driven to one
of the stable steady states $x=\pm 1$ and stays there. 
Under the natural condition, we compute
the probability density starting with a Gaussian profile 
centered at one of the stable state $x=-1$, as shown 
in Fig.~\ref{Dwell}.  
From the numerical results, we observe that, although there is no particle
concentrated around the stable position $x=1$ at the beginning, 
as time goes by, the probability density functions show two peaks 
centered at the two stable points $x=\pm1$ respectively. It suggests
 the particles tend to stay at one of two stable
points $x=\pm1$ with nearly equal chance. At the same time, 
the peaks for $\a=0.5$ are more pronounced than those for $\a=1.5$. On the
other hand, the peak at $x=1$ for $\a=1.5$ approaches 
the same height as the original peak at $x=-1$ 
faster than the case of $\a=0.5$.

%\begin{figure}[h]
%\begin{center}
%\includegraphics*[width=6.7cm,height=5.5cm]{./figure/h100_1.eps}
%\includegraphics*[width=6.7cm,height=5.5cm]{./figure/h200_1.eps}
%\end{center}
%\caption{$d=0.1$, $\eps=1$, $f=x-x^3$, $\varepsilon =1$, blue solid
%line for $\alpha=0.5$, red dashed line for $\alpha=1.5$. }
%\end{figure}

%\begin{figure}[h]
%\begin{center}
%\includegraphics*[width=6.5cm,height=4.5cm]{./figures/uniform_alpha_f_1.eps}
%\includegraphics*[width=6.5cm,height=4.5cm]{./figures/uniform_alpha_f_2.eps}
%\end{center}
%\caption{initially uniform distribution $d=0,f=x-x^3,\varepsilon
%=1$}
%\end{figure}

\section{Conclusion}
\label{sec.cl}

The Fokker-Planck equation  for a stochastic   system  with Brownian
motion (a Gaussian process) is a local differential equations, while
that for a stochastic   system with a $\alpha-$stable L\'evy motion
(a non-Gaussian process) is a nonlocal differential equation.

By exploring the Toeplitz matrix structure of the discretization, 
we have developed a fast and accurate numerical algorithm  to solve the
nonlocal Fokker-Planck equations, under either absorbing or natural
conditions. We have found the conditions under which the schemes satisfy
a maximum principle, and consequently, have shown their stability
and the convergence based on the maximum principle. 

The algorithm is tested and validated by comparing with an exact solution
and well-known numerical solutions for the probability density functions
of $\a$-stable random variables. It is applied to two stochastic 
dynamical systems (i.e., the Ornstein-Uhlenbeck system and double-well system) 
with non-Gaussian L\'evy noises. We find that the $\a$-stable process has
smoothing effect just as the standard diffusion: an initially non-smooth 
probability density function instantly becomes smooth for time $t>0$. 
The value of $\a$ has significant effect on the profiles of the density
functions.

%\medskip

\section*{Acknowledgments}
We would like to thank Zhen-Qing Chen,
Cyril Imbert, Renming Song and  Jiang-Lun Wu for helpful
discussions.
This work was partly supported by the NSF Grants 0923111 and 1025422,   and the NSFC grants 10971225 and 11028102.

\bibliographystyle{siam}      % mathematics and physical sciences
\bibliography{stochas}

%%%%%%%%%%%%%%%%%%%%%%%%%%%%%%%%%%%%%%%%%%%%%%%%%%%%%%%%%%%%%%

\end{document}